\documentclass[12pt,twoside]{amsart}
\usepackage[english]{babel} 
\usepackage[T1]{fontenc}
\usepackage[utf8]{inputenc}
\usepackage{amsmath, amssymb, amsthm}            
\usepackage{amstext, amsfonts, a4}
\usepackage{mathrsfs,enumerate,verbatim}
\usepackage[usenames,dvipsnames]{color}
\usepackage{graphics,graphicx,subfigure}
\usepackage{caption} 
\usepackage{xcolor}

\usepackage{lmodern}
\usepackage{xfrac}
\usepackage{faktor}
\theoremstyle{plain}
	\newtheorem{Thm}{Theorem}[section] 
	\newtheorem{Prop}[Thm]{Proposition}        
	\newtheorem{Lem}[Thm]{Lemma}           
	\newtheorem{Coro}[Thm]{Corollary}
	\newtheorem*{Warn}{Warning}

\theoremstyle{definition}
	\newtheorem{Def}[Thm]{Definition}
    
	\newtheorem{Nota}[Thm]{Notation}

\theoremstyle{remark}
	\newtheorem{Rem}[Thm]{Remark}
	\newtheorem{Que}[Thm]{Question}

\def\emptyset{\varnothing}

\def\NN{{\mathbf N}}
\def\ZZ{{\mathbf Z}}
\def\RR{{\mathbf R}}
\def\QQ{{\mathbf Q}}

\DeclareMathOperator\GL{GL}
\DeclareMathOperator\Mon{Mon}

\def\cB{{\mathcal B}}

\def\cF{{\mathcal F}}

\def\cP{{\mathcal P}}
\def\cQ{{\mathcal Q}}
\def\cR{{\mathcal R}}

\newcommand{\Ker}{\operatorname{Ker}}
\newcommand{\Id}{\operatorname{Id}}
\newcommand{\eps}{\varepsilon}
\newcommand{\support}{\operatorname{Supp}}
\newcommand{\Rec}{\operatorname{Rec}}
\newcommand{\SAF}{\operatorname{SAF}}
\newcommand{\IET}{\operatorname{IET}}
\newcommand{\Vect}{\operatorname{Vect}}
\newcommand{\Card}{\operatorname{Card}}
\newcommand{\GtG}{\operatorname{GtG}}
\newcommand{\Transpo}{\mathscr{T}}
\newcommand{\interfo}[2]{\mathopen{[} #1,#2 \mathclose{[}}
\newcommand{\pr}{\operatorname{pr}}
\newcommand{\ab}{\mathrm{ab}}
\newcommand{\whei}{\operatorname{WHei}}

\newcommand{\vl}{\operatorname{vol}^\otimes}

\newcommand{\Grd}{\operatorname{Grd}}
\newcommand{\City}{\operatorname{City}}
\newcommand{\Sky}{\operatorname{Sky}}
\newcommand{\Top}{\operatorname{Top}}
\newcommand{\Site}{\operatorname{Site}}
\newcommand{\Work}{\operatorname{Work}}

\newcounter{saveenum}
\begin{document}

\title{On groups of rectangle exchange transformations}
\author{Yves Cornulier and Octave Lacourte}
\date{\today}

\begin{abstract}
We study a generalization $\Rec_d$ of the group IET of interval exchange transformations in every dimension $d \geq 1$, called the rectangle exchange transformations group. The subset of restricted rotations in $\IET$ is a generating subset and we prove that a natural generalization of these elements, called restricted shuffles, form a generating subset of $\Rec_d$. We denote by $\mathscr{T}_d$ the subset of $\Rec_d$ made up of those transformations that permute two rectangles by translations. We prove that the derived subgroup is generated by $\mathscr{T}_d$. We also identify the abelianization of $\Rec_d$.
\end{abstract}
\maketitle

\setcounter{tocdepth}{1}
\tableofcontents

\begin{figure}[h!]
\begin{center}
\includegraphics[scale=0.4]{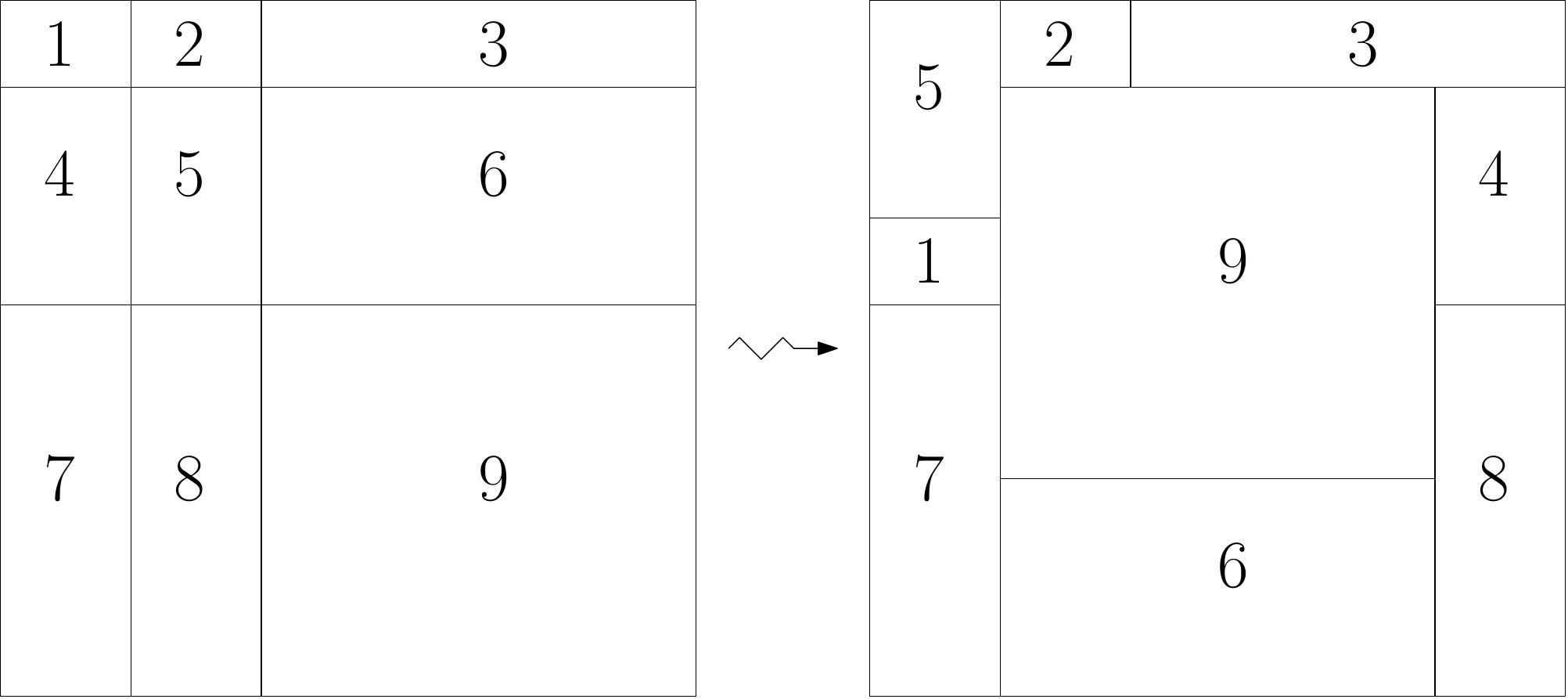}
\end{center}\caption{A rectangle exchange transformation in dimension $d=2$.}
\label{mainfig}
\end{figure}

\section{Introduction}

Let us define a {\bf rectangle exchange transformation} as an invertible self-trans\-for\-mation of the cube $[0,1\mathclose[^d$ that consists in cutting the square into finitely many rectangles and moving these rectangles by translations to get another partition of the square (see Figure \ref{mainfig}). See Section \ref{precise} for a rigorous definition.

For $d=1$, this reduces to the widely studied group $\IET$ of interval exchange transformations.

Historically, H. Haller \cite{Haller} introduced $2$-rectangle exchange transformations in 1981 and it is mainly ergodic properties of a single $2$-rectangle exchange transformation which are studied. More generally, dynamics of piecewise isometries on polytopes are studied, in particular by A. Goetz \cite{Goetz}, however the group itself is rarely considered. The larger groups of piecewise affine self-homeomorphisms of some affine manifolds were recently considered in particular by D. Calegari and D. Rolfsen \cite{Calegari-Rolfsen}.

Here our goal is to initiate the study of $\Rec_d$ as a group, beyond the case $d=1$. Our main results describe the abelianization homomorphism and establish that the derived subgroup is a simple group. Such results make use of the description of suitable generating subsets, which are also interest for their own sake.

We introduce two kinds of special elements in $\Rec_d$ (see Figure \ref{fig_shuf} for pictures and Definition \ref{Definition restricted shuffle and REC-transposition} for rigorous definitions).

\begin{Def}
A {\bf restricted shuffle} (depicted in Figure \ref{fig_shuf}) is an element of $\Rec_d$ that is identity outside some rectangle $R_1\cup R_2$, where $R_1$ and $R_2$ are ``consecutive'' rectangles (have disjoint interior and share a common facet), and ``shuffles'' $R_1$ and $R_2$.


A {\bf rectangle transposition} is the map, given (interior-)disjoint rectangles $R_1,R_2$ that are translates of each other, exchanges them by translation, and is identity elsewhere.


\end{Def}

\begin{figure}[h!]
\begin{center}
\includegraphics[width=0.45\textwidth]{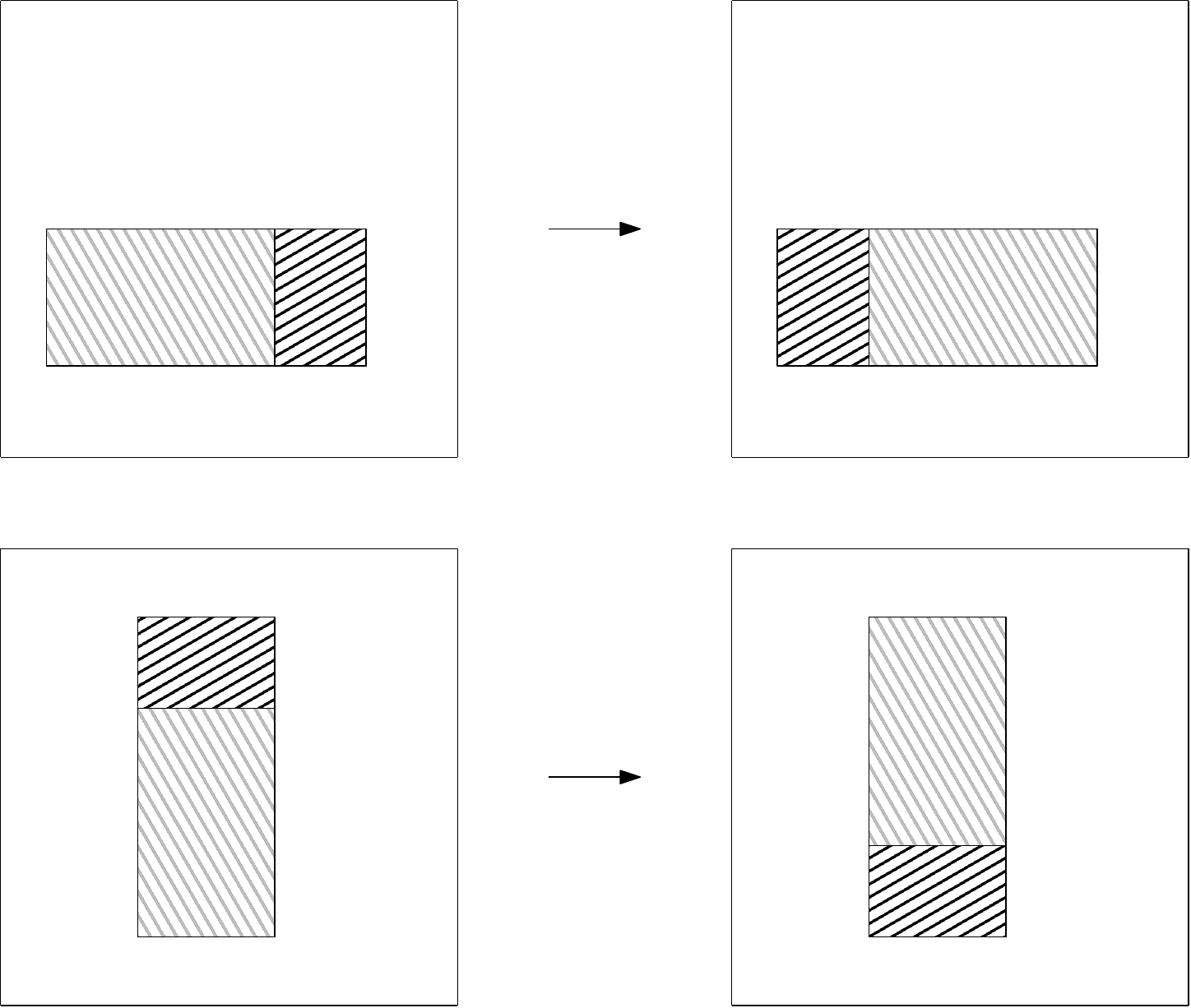} \hspace{20pt}
\includegraphics[width=0.38\textwidth]{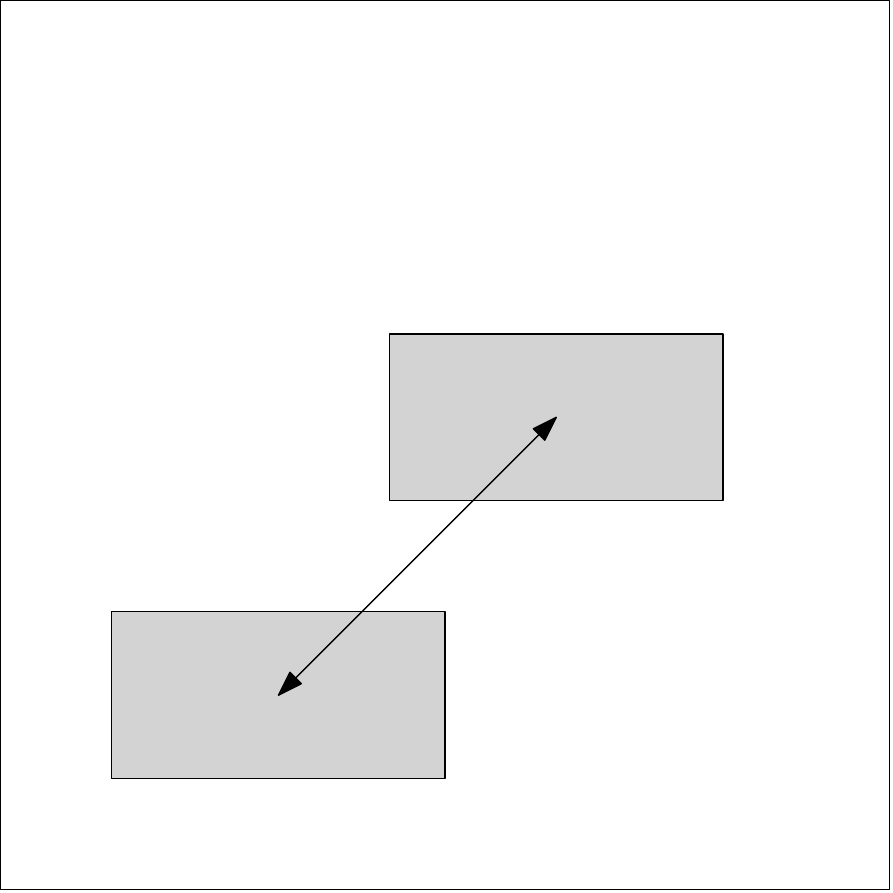}
\caption{\textbf{Left}: Examples of restricted shuffles in dimension 2 in both directions. 
\textbf{Right}: Example of a rectangle transposition in $\Rec_2$.}\label{fig_shuf}
\end{center}
\end{figure}

\begin{Thm}\label{Theorem the set of all restricted shuffles is a generating subset of REC}\label{rec_shu_gen}
The set of all restricted shuffles is a generating subset of $\Rec_d$.
\end{Thm}

For $d=1$, restricted shuffles are known as restricted rotations. It is a well-known observation that they form a generating subset of $\IET$: after encoding an interval exchange transformation as a permutation with given interval lengths, this is an easy consequence of the fact that the symmetric group $\mathfrak{S}_n$ is generated by transpositions $(i,i+1)$ for $1\le i<n$. This argument falls apart for $d\ge 2$, as the combinatorics of a rectangle exchange is not always well-encoded by a permutation, and conversely because rearranging rectangles does not always define a rectangle exchange. The proof of Theorem \ref{Theorem the set of all restricted shuffles is a generating subset of REC} is indeed significantly more involved. 
For $d=2$ a variant of the proof, providing a combinatorial refinement of Theorem \ref{Theorem the set of all restricted shuffles is a generating subset of REC}, is performed in Section \ref{Section A refinement for REC 2}.

The various next results actually make a crucial use of Theorem \ref{rec_shu_gen}.

Thanks to Theorem \ref{Theorem the set of all restricted shuffles is a generating subset of REC} we obtain that the derived subgroup $D(\Rec_d)$ is generated by conjugates of commutators of two restricted shuffles. With this result we prove the following theorem:

\begin{Thm}\label{Theorem D(REC) is generated by Transpo}
The derived subgroup $D(\Rec_d)$ is simple and generated by its subset of rectangle transpositions. It is contains every nontrivial normal subgroup of $\Rec_d$.
\end{Thm}

Arnoux-Fathi and independently Sah exhibited a surjective homomorphism from $\IET=\Rec_1$ onto the abelian group $\Lambda^2_\QQ\RR$, the second exterior algebra of $\RR$ over $\QQ$, which is now known as SAF homomorphism. Moreover, Sah proved (see \cite{Arnoux,Sah}) that it induces an isomorphism from the abelianization of $\IET$ onto $\Lambda^2_\QQ\RR$.

Our next contribution is to exhibit the suitable analogue of the SAF homomorphism in the context of $\Rec_d$. We denote by $\RR^{\otimes k}$ the $k$-th tensor power of $\RR$ over $\QQ$. 

\begin{Thm}\label{Theorem Generalized SAF}
There is a natural surjective group homomorphism from $\Rec_d$ onto $( \RR^{\otimes (d-1)} \otimes (\bigwedge^2_{\QQ} \RR))^d$, called the generalized $\SAF$-homomorphism, whose kernel is the derived subgroup $D(\Rec_d)$.
\end{Thm}

Let us partially describe this abelianization homomorphism here. In $\RR^d$, define a {\bf rectangle} as a product $\prod_{i=1}^d[a_i,b_i\mathclose[$ of left-closed right-open bounded intervals. Define a {\bf multirectangle} as a finite union of rectangles. 
We define the tensor volume $\vl_d(M)$ of a multirectangle in $\RR^d$ as follows: $\vl_d(\prod_{i=1}^d[a_i,a_i+t_i\mathclose[)=t_1\otimes\dots\otimes t_d\in\RR^{\otimes d}$ (tensor product over $\QQ$), and $\vl_d$ is additive under disjoint unions. This is well-defined, by a simple argument (Proposition \ref{tensor_measure}).

Then the abelianization homomorphism can essentially be described as a (non-surjective) homomorphism $\tau=(\tau_1,\dots,\tau_d)$ into $(\RR^{\otimes (d+1)})^d$, where $\tau_i$ is defined by 
\[\tau_i(f) = \sum_{x\in\RR^d}\vl_d\big((f-\mathrm{id})^{-1}(\{x\})\big)\otimes x_i.\]
The main substance of the proof consists in proving that every element in the kernel of $\tau$ is a product of commutators.
We have found it convenient to rewrite the proof of the original IET case ($d=1$), using Lemma \ref{uab} which identifies some abelian group defined by a suitable infinite presentation. This approach allows to avoid relying on too many computations, and especially performs these computations in a context disjoint from IET. The general case then relies on an elaboration of this combinatorial algebraic lemma. The image of $\tau$ is easy to described (and coincides with $\bigoplus_{i=1}^d\mathrm{Im}(\tau_i)$); we use a simple change of coordinates to describe it more smoothly in \S\ref{gsaf}.

As an application, in Section \ref{Section a normal subgroup bigger than the derived subgroup} we consider the subgroup $\GtG_d$ of $\Rec_d$ generated by the subset $\IET^d \cup \Transpo_d$ (where the group $\IET^d$ acts coordinate-wise and $\Transpo_d$ is the set of rectangle transpositions). Obviously $\GtG_1=\Rec_1$. In contrast, a consequence of Theorem \ref{Theorem Generalized SAF} (along with the description of the abelianization homomorphism) for $d \geq 2$ is:

\begin{Coro}
The group $\GtG_d$ is a proper normal subgroup of $\Rec_d$, which strictly contains $D(\Rec_d)$.
\end{Coro}

For $d\ge 1$, let $\Rec_d^{\bowtie}$ be the group of ``rectangle exchanges with flips'' (acting on rectangles with piecewise isometries whose linear part are diagonal with $\pm 1$ diagonal entries, see \S\ref{sflip}).

\begin{Coro}\label{recbowsimple}
For every $d\ge 1$, the group $\Rec_d^{\bowtie}$ is a simple group.
\end{Coro}

This is known for $d=1$ (proved in Arnoux's thesis \cite{ArnouxThese} and reproduced in the appendix of \cite{GuelmanLiousse2021}). The result in general follows with little effort from the previous results.

An observation (Proposition \ref{ori_lift}) is that $\Rec_d^{\bowtie}$ can be embedded in $\Rec_d$. A general construction of Nekrashevych can be used in this context to yield the following (see Section \ref{torsion_section}).

\begin{Thm}
There exists an explicit infinite finitely generated subgroup of $\Rec_3^{\bowtie}$ (and hence of $\Rec_3$) that is infinite and torsion.
\end{Thm}

It is explicit in the sense that it is generated by 3 explicit elements of order 2.

When $d=2$ do not know whether $\Rec_d$ contains any infinite finitely generated torsion subgroup (for $d=1$ this is a well-known open question).

Let us now consider a slightly more general, and more natural, framework. If $M$ is a multirectangle in $\RR^d$, define $\Rec_d(M)$ as the group of rectangle exchange self-transformations of $M$. Thus $\Rec_d=\Rec_d([0,1\mathclose[^d)$. For $d=1$, it is straightforward to see that all such groups are isomorphic (for $M$ nonempty). In dimension $\ge 2$ this is probably false. (See also \S\ref{rec_torus} for $\Rec_d(T)$ when $T$ is a torus, that is, the quotient of $\RR^d$ by a lattice.)

Note that even in the study of $\Rec_d$, such groups unavoidably appear: if $M\subset [0,1\mathclose[^d$ is a multirectangle, the group $\Rec_d(M)$ can be viewed as a subgroup of $\Rec_d$, namely those elements that are identity outside $M$. Note that for $M,M'$ homothetic, $\Rec_d(M)$ and $\Rec_d(M')$ are isomorphic. Hence, for any nonempty multirectangles $M,M'$ in $\RR^d$, the groups $\Rec_d(M)$ and $\Rec_d(M')$ embed into each other.

\begin{Coro}
For every $d\ge 1$ and nonempty multirectangle $M$ in $\RR^d$, the group $\Rec_d(M)$ has a simple derived subgroup.
\end{Coro} 

\begin{Coro}
For every $d\ge 1$ and multirectangle $M$ in $\RR^d$ with connected interior, the group $\Rec_d(M)$ is generated by restricted rotations.
\end{Coro}

(Note that connectedness of the interior is an obvious necessary condition.)

Let us now pass to mostly open questions.

A natural question is to classify the groups $\Rec_d(M)$ up to isomorphism.

 The action of $\Rec_d(M)$ preserves the tensor volume of multirectangles, and actually this determines orbits of the action on the set of sub-multirectangles:

\begin{Prop}\label{vol_orb}
For multirectangles $M_1,M_2\subseteq M$ there exists $f\in\Rec_d(M)$ such that $f(M_1)=M_2$ if and only of $\vl_d(M_1)=\vl_d(M_2)$.
\end{Prop}

Define the monomial group as the subgroup of $\GL_d(\RR)$ generated by diagonal and permutation matrices, and denote it by $\Mon_d(\RR)$ (it is isomorphic to ${\RR^*}^d\rtimes \mathfrak{S}_d$). For $g\in\Mon_d(\RR)$, one easily sees that $g$ conjugates $\Rec_d(M)$ to $\Rec_d(g(M))$. Combining with the previous proposition, we deduce:

\begin{Prop}
For multirectangles $M_1,M_2$ in $\RR^d$, if $\vl_d(M_1)$ and $\vl_d(M_2)$ are in the same orbit under the canonical $\Mon_d(\RR)$-action on $\RR^{\otimes d}$, then $\Rec_d(M_1)$ and $\Rec_d(M_2)$ are isomorphic.
\end{Prop} 

Our main open question is whether the converse holds.

\begin{Que}\label{isobrec}
Conversely, for nonempty multirectangles $M_i\in \RR^{d_i}$, $i=1,2$, if $\Rec_{d_1}(M_1)$ and $\Rec_{d_2}(M_2)$ are isomorphic, does it follow that $d_1=d_2$ and $\vl_d(M_1)=\vl_d(M_2)$?
\end{Que}

If we only focus on the dimension issue, one can ask about a stronger rigidity:

\begin{Que}
For nonempty multirectangles $M_i\in \RR^{d_i}$, $i=1,2$, if $\Rec_{d_1}(M_1)$ (or its derived subgroup) embeds as a subgroup of $\Rec_{d_2}(M_2)$, does it follow that $d_1\le d_2$?
\end{Que}

Question \ref{isobrec} asks about the existence of isomorphisms. The following asks about a precise description of isomorphisms, and would imply a positive answer to Question \ref{isobrec}.

\begin{Que}\label{isomo_rec}
For multirectangles $M_1,M_2$ in $\RR^d$, and an isomorphism $f:\Rec_d(M_1)\to \Rec_d(M_2)$, does there exist $g\in\Mon_d(\RR)$ such that, $g_*$ denoting the isomorphism $\Rec_d(M_2)\to\Rec_d(g(M_2))$ induced by $g$, the composite map $g_*\circ f$ is induced by a Rec-isomorphism from $M_1$ into $g(M_2)$?
\end{Que}

Rubin's theorem \cite[Corollary 3.5]{Rubin} ensures that each such isomorphism is induced by conjugation by a homeomorphism $\bar{M}_1\to\bar{M}_2$. Here, roughly, $\bar{M}$ denotes $M$ with each point blown-up to $2^n$ points (one choice for each direction). More precisely, $\bar{\RR}$ means $\RR$ with each point $x$ replaced with a pair $\{x_-;x_+\}$, with the order topology. Then $\bar{M}\subset\bar{\RR}^d$ means the interior of the set of points mapping to the closure of $M$. So the question is whether such a homeomorphism is necessarily composition of a Rec-isomorphism and a monomial map. This question is not only relevant to classify the groups $\Rec_d(M)$ up to isomorphism (when $M$ varies), but also (when $M_1=M_2=M$) to understand the automorphism group of the groups $\Rec_d(M)$. Question \ref{isomo_rec} has a positive answer when $d=1$, where essentially it asserts that the outer automorphism group of $\IET$ has order 2, a result of Novak~\cite{Novak}.

Another question would be to describe a presentation of $\Rec_d$ using the set of restricted shuffles as set of generators. Of course this question is imprecise, or has a trivial answer: take all relations as set of relators. The point is rather to exhibit a natural family of relators. Specifically, we can ask the following, which even for $d=1$ is unknown:

\begin{Que}
Is $\Rec_d$ boundedly generated over the set of restricted shuffles? That is, does there exists $N$ such there is presentation of $\Rec_d$ with all restricted shuffles as set of generators, and relators of length $\le N$?
\end{Que}

A possible motivation for exhibiting ``nice'' presentations would be to solve the following, even for $d=1$:

\begin{Que}
What is the second homology group $H_2(\Rec_d)$? is $H_2(D(\Rec_d))$ reduced to $\{0\}$?
\end{Que}

Last and not least, let us ask:

\begin{Que}
Is $\Rec_d$ amenable? Does it fail to contain any non-abelian free subgroup?
\end{Que}

For $d=1$ these are well-known questions; the amenability question is raised in \cite{CorBou} and the question of (non)-existence of a free subgroup is due to A.~Katok. In this direction, let us mention the easy:

\begin{Prop}\label{recfm}
The group $\Rec_d$ has no infinite subgroup with Property FM. In particular, it has no infinite subgroup with Kazhdan's Property T.
\end{Prop}

See Section \ref{s_kaz} for the short proof. Recall that a group $\Gamma$ has Property FM \cite{cornulier2015irreducible} if every $\Gamma$-set with an invariant mean has a finite orbit. In particular, an amenable group with Property FM has to be finite. Proposition \ref{recfm} was obtained in the IET case ($d=1$) in \cite[Theorem 6.1]{DFG2013} (stated for Property T, but using only Property FM).

\medskip

{\bf Outline.} The main definitions are given in \S\ref{precise}. In \S\ref{s_tensor_volume} we establish some basic facts based on a classical theorem of Eliott about totally ordered abelian groups. We then proceed to the proof of Theorem \ref{rec_shu_gen}: after some easy cases in \S\ref{gen_easy}, we describe the general procedure in \S\ref{Section Generation by restricted shuffles} (eventually boiling down to these easy basic cases). In \S\ref{Section A refinement for REC 2}, we prove a ``jigsaw'' refinement of Theorem \ref{rec_shu_gen} in dimension 2. We extend Theorem \ref{rec_shu_gen} to multirectangles in \S\ref{gen_multir}, exhibiting, along the way, some maximal subgroups of $\Rec_d$. In \S\ref{Section The derived subgroup}, we notably prove simplicity of the derived subgroup of $\Rec_d$. In \S\ref{Section Abelianization of REC}, we describe the generalized SAF-homomorphism and prove that it is indeed the abelianization homomorphism. In \S\ref{sflip}, we address rectangle exchanges with flips, notably establishing the simplicity of this group. Finally, in \S\ref{torsion_section}, we exhibit a infinite, finitely generated torsion subgroup in $\Rec_3$.

\section{Precise main definitions}\label{precise}

We recall that the group $\IET$ is the group consisting of all permutations of $\interfo{0}{1}$ continuous outside a finite set, right-continuous and piecewise a translation. 

We study a generalization of $\IET$ in higher dimension. Let $d \geq 1$ be an integer. We denote by $X=\interfo{0}{1}^d $ the left half-open square of dimension $d$. Let $\cB= \lbrace e_1, e_2, \ldots , e_d \rbrace$ be the canonical basis of $\RR^d$ and we denote by $\lambda$ the Lebesgue measure on $\RR$. For $1 \leq i \leq d$, let $\pr_i$ be the orthogonal projection on $\Vect(e_i)$ and $\pr_i^{\bot}$ be the orthogonal projection on the hyperplane $e_i^{\bot}$. For an element $x \in \RR^d$ we use the notation $x_i=\pr_i(x)$. A natural way to generalize left half-open intervals is to consider elements of the form $I_1 \times \ldots \times I_d$ where $I_i$ is a left half-open subinterval of $\interfo{0}{1}$. They are called left half-open $d$-rectangles. In the following, every $d$-rectangle is supposed to be left half-open.

We define the rectangle exchange transformations group of dimension $d$, denoted by $\Rec_d$, as the set of all permutations $f$ of $\interfo{0}{1}^d$ such that there exists a finite partition of $\interfo{0}{1}^d$ into $d$-rectangles such that $f$ is a translation on each of these $d$-rectangles. Elements of $\Rec_d$ are called $d$-{\bf rectangle exchange transformations}.

In the sequel, all partitions are meant to be finite. The simplest partitions into rectangles are the following:

\begin{Def}
A partition $\cP$ of $\interfo{0}{1}^d$ into rectangles is called a {\bf grid-pattern} if for every $1 \leq i \leq d$, there exists a partition $\cQ_i$ of $\mathopen{[} 0,1 \mathclose{[}$ into half-open intervals such that $\cP= \cQ_1 \times \cQ_2 \times \ldots \times \cQ_d$.
\end{Def}

Obviously, every partition of $\interfo{0}{1}^d$ a rectangle into rectangles can be refined into a grid pattern.

\begin{Def}
Let $f \in \Rec_d$ and $\cP$ be a partition of $X$ into rectangles. We say that $\cP$ is a {\bf partition associated with $f$} if for every $K \in \cP$ the restriction of $f$ to $K$ is a translation. Then the set $f(\cP):=\lbrace f(K) \mid K \in \cP \rbrace$ is a new partition of $X$ into rectangles called the {\bf arrival partition} of $f$ with $\cP$. We denote by $\Pi_f$ the set of all partitions associated with $f$. If $\cP$ is a grid-pattern, it is said to be a {\bf grid-pattern associated with $f$}.
\end{Def}

\begin{Rem}
The fact that $\Rec_d$ is a group under composition is immediate. One can see that if $f,g \in \Rec_d$ and $\cP \in \Pi_f,\cQ \in \Pi_g$, then there exists a partition $\cR$ into $d$-rectangles that refines both $f(\cP)$ and $\cQ$. Thus $f^{-1}(\cR)$ is a partition into $d$-rectangles such that $g \circ f$ acts on every $d$-rectangle of $f^{-1}(\cR)$ by translation.
\end{Rem}

In the following, the ``$d$'' of $d$-rectangle may be omitted whenever there is no possible confusion.

\begin{Def}\label{Definition restricted shuffle and REC-transposition}
(See Figure \ref{fig_shuf} in the introduction.)

A {\bf restricted shuffle in direction $i$} is an element $\sigma_{R,s,i}$ of $\Rec_d$ where $R$ is a $(d-1)$-subrectangle of $e_i^{\bot}$ and $s$ is a restricted rotation, defined by:
\begin{enumerate}
\item if $\pr_i^\bot(x)\notin R$, $\sigma_{R,s,i}(x)=x$;
\item if $\pr_i^\bot(x)\in R$:
        \begin{enumerate}
        \item for $j\neq i$, $\sigma_{R,s,i}(x)_j=x_j$;
        \item $\sigma_{R,s,i}(x)_i=s(x_i)$.
        \end{enumerate}
\end{enumerate}

For disjoint translation-isometric rectangles $P,Q \subset [0,1[^d$, define the {\bf rectangle transposition} $\tau_{P,Q}$ as the element of $\Rec_d$ defined as the identity outside $P \cup Q$, and as a translation on each of $P,Q$, exchanging them. The set of all rectangles transpositions in $\Rec_d$ is denoted by $\Transpo_d$.

\end{Def}

\begin{Nota}
If $I$ and $J$ are the two intervals associated with $s$ then the $d$-rectangles $P_1$ and $P_2$, defined by $\pr_i(P_1)=I$, $\pr_i(P_2)=J$ and $\pr_i^{\bot}(P_1)=\pr_i^{\bot}(P_2)=R$, are two rectangles which partitioned the support of $f$ and where $f$ is continuous on both of them. We say that $f$ shuffles this two rectangles.
\end{Nota}

\section{Setwise freeness, Eliott's theorem and the tensor volume}\label{s_tensor_volume}

\subsection{Eliott's theorem}

At various places, we need a general fact on totally ordered abelian groups. A submonoid $S$ of an abelian group is said to be {\bf simplicial} if it is generated, as a submonoid, by a finite $\ZZ$-independent subset. It is said to be {\bf ultrasimplicial} if every finite subset of $S$ is contained in a simplicial submonoid of $S$. An ordered abelian group is said to be {\bf ultrasimplicially ordered} if its positive cone (the submonoid of elements $\ge 0$) is ultrasimplicial.

\begin{Thm}[Eliott \cite{Elliott79}]\label{Theorem ultrasimplicially ordered group}
Every totally ordered abelian group is ultrasimplicially ordered.
\end{Thm}

(For real numbers, this statement was rediscovered as Lemma 4.1 of Vorobets in \cite{Vorobets2011}, who was the first to use it in the context of interval exchanges.)

\subsection{Setwise $\QQ$-freeness}

We fix $d \geq 1$; in a first reading, one can assume $d=2$.

In fact we will need some rigidity on partitions associated with an element of $\Rec_d$. For this we want to have some objects to be $\QQ$-free.

\begin{Def}
Let $\cP$ be a partition into rectangles of $\interfo{0}{1}^d$. For every $1 \leq i \leq d$ we denote by $\cF_i$ the set $\lbrace \lambda(\pr_i(K)) \mid K \in \cP \rbrace$. If for every $1 \leq i \leq d$ the set $\cF_i$ is $\QQ$-linearly independent then we say that $\cP$ is a {\bf setwise $\QQ$-free} partition.
\end{Def} 

\begin{Warn} The required $\QQ$-independence  is that of the {\bf set} $ \lbrace \lambda(\pr_i(K)) \mid K\in \cP \rbrace$, and not the family $(\lambda(\pr_i(K)))_{K \in \cP}$. So the setwise freeness condition says, roughly speaking, that the only $\QQ$-linear dependence relations among the $\lambda(\pr_i(K))$, for $K \in \cP$ (for each fixed $K$) are equalities.
\end{Warn}

The previous warning, as well as the following proposition are illustrated in Figure \ref{fig_refi}.

\begin{Prop}\label{Proposition We can always refine a grid-pattern into a setwise Q-free grid-pattern}
Let $\cQ$ be a grid-pattern. There exists a setwise $\QQ$-free grid-pattern $\cQ'$ that refines $\cQ$.
\end{Prop}

\begin{proof}
Write $\cQ=\cQ_1 \times \ldots \times \cQ_d$ where $\cQ_i$ is a partition into intervals of $\interfo{0}{1}$ and let $\cF_i:=\lbrace \lambda(I) \mid I \in \cQ_i \rbrace$. By Theorem \ref{Theorem ultrasimplicially ordered group}, there exists a $\QQ$-free subset $\cF_i'$ of positive reals such that every element of $\cF_i$ belongs to the additive subsemigroup generated by $\cF'_i$. Hence we can refine each $\cQ_i$ as a partition $\cQ'_i$. Then $\cQ':= \cQ_1' \times \ldots \times \cQ_d'$ is a setwise $\QQ$-free grid-pattern which refines $\cQ$.
\end{proof}

\begin{figure}[h!]
\begin{center}
\includegraphics[width=\textwidth]{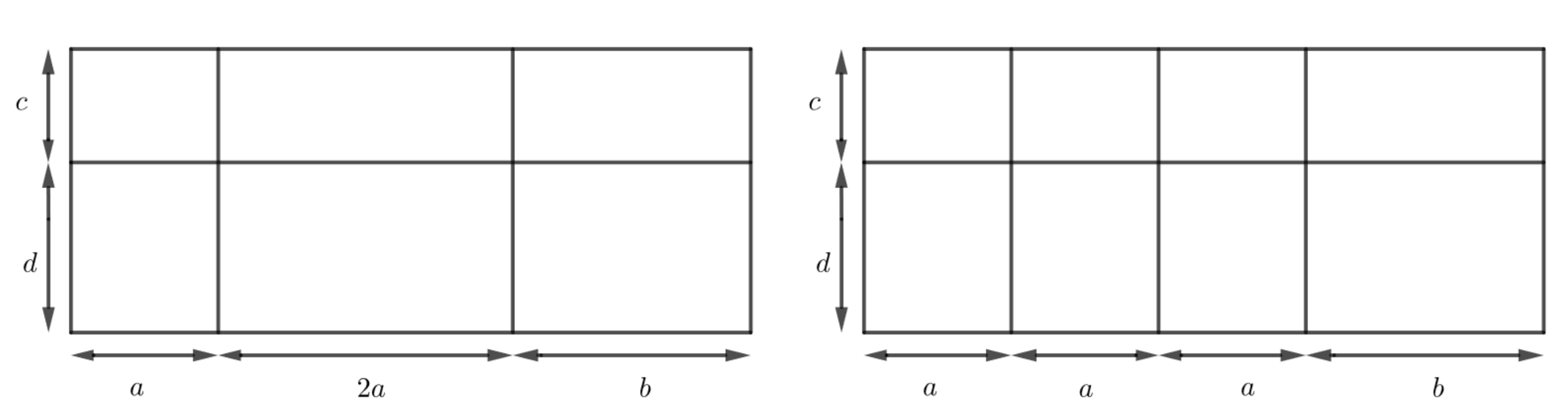}
\caption{\textbf{Left:} A grid-pattern that is not setwise $\QQ$-free. \textbf{Right:} A setwise $\QQ$-free grid-pattern which refines the left-hand grid-pattern. (We assume that $\lbrace a,b \rbrace$ and $\lbrace c,d \rbrace$ and setwise $\QQ$-free subsets of $\RR$.)}\label{fig_refi}
\end{center}
\end{figure}

\subsection{Tensor volume}

\begin{Prop}[Tensor volume]\label{tensor_measure}
There is a unique map $\vl_d$ from the set of multirectangles in $\RR^d$ to $\RR^{\otimes d}$ that is additive under disjoint unions, and maps each rectangle $\prod_{i=1}^d [a_i,a_i+t_i\mathclose[$ to $t_1\otimes\dots\otimes t_d$.
\end{Prop}
\begin{proof}
The uniqueness is clear since every multirectangle is a disjoint union of rectangles.

For the existence, we first check additivity when a rectangle is decomposed onto rectangles according to partitions in each direction (call this a regular partition). This is straightforward from multilinearity. Next, we need to show that if a multirectangle is a finite disjoint union of rectangles in two ways then the resulting computation of its tensor measure yields the same result. Indeed, there exists a common refinement of these two partitions that is a regular partition of each of the rectangles in both partition. Hence, the equality follows from the above particular case of additivity.
\end{proof}

Proposition \ref{vol_orb} follows from the following:

\begin{Lem}\label{vl_rec}
Let $R,R'$ be multirectangles in $\RR^d$. Then there exists a $\Rec$-isomorphism $R\to R'$ if and only if $\vl_d(R)=\vl_d(R')$.
\end{Lem}
\begin{proof}
Since $\vl_d$ is preserved by translations and is additive under disjoint unions, it is preserved by REC-isomorphisms, so the condition is necessary. Conversely, suppose $\vl_d(R)=\vl_d(R')$. Choose finite partitions of $R$ and $R'$ into rectangles, called constituting rectangles. Let $H_i$ be the subsemigroup of $\RR$ generated by $i$-sizes of constituting rectangles. By Eliott's theorem (Theorem \ref{Theorem ultrasimplicially ordered group}), $H_i$ is contained in the subsemigroup generated by some $\QQ$-free subset $S_i$ of $\RR_{>0}$. Hence, refining the partitions, we can suppose that all $i$-sizes of constituting rectangles are in $S_i$. For $s=(s_1,\dots,s_d)\in S_1\times\dots\times S_d$, let $n(s)$ (resp.\ $n'(s)$) be the number of rectangles in $R$ (resp.\ $R'$) of size $(s_1,\dots s_d)$. 
Also write $\bar{s}=s_1\otimes \dots\otimes s_d$. Then $\sum_sn(s)\bar{s}=\vl_d(R)=\vl_d(R')=\sum_sn'(s)\bar{s}$. Since the $\bar{s}$ form a $\QQ$-free family when $s$ ranges over $S_1\times\dots\times S_d$, we deduce that $n(s)=n'(s)$ for all $s$. Hence there is a shape-preserving bijection between the set of constituting rectangles of $R$ and $R'$. Such a bijection induces a REC-isomorphism $R\to R'$.
\end{proof}





The last part of the argument also provides the following statement about partitions of a given multirectangle, which will be used in the sequel.

\begin{Lem}\label{Lemma two partitions with Q-conditions contains the same number of pieces}
Let $R$ be a multirectangle in $\RR^d$. For every $1 \leq i \leq d$, let $F_i$ be a setwise $\QQ$-free subset of $\RR^+$. Let $\cP$ and $\cP'$ be two partitions into $d$-rectangles of $R$ such that for every $K \in \cP \cup \cP'$ we have $\lambda(\pr_i(P)) \in F_i$. Then, there exists a bijection $\delta$ between $\cP$ and $\cP'$ such that for every $K \in \cP$, the rectangles $K$ and $\delta(K)$ are translation-isometric. If $K \in \cP \cap \cP'$ we can also ask $\delta(K)=K$.
\end{Lem}
\begin{proof}
Let $R'$ be the union of $\cP\cap\cP'$. Replacing $R$ with $R\smallsetminus R'$, we can suppose that $\cP\cap\cP'$ is empty (and act as identity on common rectangles), and hence ignore the last requirement.

The sequel is similar to the proof of Lemma \ref{vl_rec}.
For $s=(s_1,\dots,s_d)\in S_1\times\dots\times S_d$, let $n(s)$ (resp.\ $n'(s)$) be the number of rectangles in $\cP$ (resp.\ $\cP'$) of size $(s_1,\dots s_d)$, and write $\bar{s}=s_1\otimes\dots\otimes s_d$. Since the $\bar{s}$ form a $\QQ$-free subset, we deduce that $n(s)=n'(s)$ for every $s$. Hence there is a bijection as required.
\end{proof}

\section{Generation by restricted shuffles: first observations}\label{gen_easy}

We establish some easy particular cases of Theorem \ref{Theorem the set of all restricted shuffles is a generating subset of REC}, which asserts that $\Rec_d$ is generated by restricted shuffles.

We start with the well-known case $d=1$:

\begin{Prop}\label{IETrr}
The group $\IET$ is generated by restricted rotations.
\end{Prop}
\begin{proof}
The symmetric group $\mathfrak{S}_n$ is generated by the transpositions $(i\;i+1)$, $1\le i<n$.
Each IET can be viewed as a permutation of intervals, and therefore this group is generated by those ones consisting in transposing two consecutive intervals. These are precisely restricted rotations.
\end{proof}

A direct consequence of the definition of a restricted shuffle, Definition \ref{Definition restricted shuffle and REC-transposition}, and Proposition \ref{IETrr} is the following proposition, which is a first easy particular case of Theorem \ref{Theorem the set of all restricted shuffles is a generating subset of REC}, and a step in its proof.

\begin{Prop}\label{Proposition every element of IET^d is a product of restricted shuffles}
Every element of $\IET^d$ is a finite product of restricted shuffles. \qed
\end{Prop}

Here is a second elementary particular case of Theorem \ref{Theorem the set of all restricted shuffles is a generating subset of REC}, which will also be needed.

\begin{Prop}\label{Proposition any transposition is a product of restricted shuffles}
For all disjoint translation-isometric $P,Q$ rectangles, the rectangle transposition $\tau_{P,Q}$ is a product of restricted shuffles.
\end{Prop}

\begin{proof}
We first prove this in the special case when there exists $1 \leq i \leq d$ such that $\pr_i (P) \cap \pr_i (Q) = \emptyset$ and $\pr_i^{\bot}(P)=\pr_i^{\bot}(Q)$. In this case we obtain it is a product of two restricted shuffles. Indeed, this is a consequence of the fact that this lemma is true when $d=1$. Let $a,b,a',b' \in \interfo{0}{1}$ such that $\pr_i (P)= \interfo{a}{b}$ and $\pr_i (Q)=\interfo{a'}{b'}$. Up to change the role of $P$ and $Q$ we can assume that $b<a'$. Let $R$ and $S$ be the two rectangles such that $\pr_i^{\bot} (R)=\pr_i^{\bot}(S)= \pr_i^{\bot} (P)$ and $ \pr_i(R)=\interfo{b}{b'}$ and $\pr_i(S)=\interfo{b}{a'}$. Let $r_1$ be the restricted shuffle in direction $i$ that shuffles $P$ with $R$ (this one send $P$ on $Q$) and $r_2$ be the restricted shuffle in direction $i$ that permutes $P$ with $S$. Then the composition $r_2^{-1} r_1$ is equals to the rectangle transposition that permutes $P$ with $Q$.

Now let us prove the general case. Let $P$ and $Q$ be two rectangles which are translation-isometric such that $P \cap Q=\emptyset$. Let $P_i:=\pr_i(P)$ and $Q_i:=\pr_i(Q)$ for every $1 \leq i \leq d$. Thus $P=P_1 \times P_2 \times \ldots \times P_d$ and $Q=Q_1 \times Q_2 \times \ldots \times Q_d$. For every $1 \leq i \leq d-1$ let $R_i$ be the rectangle $Q_1 \times \ldots \times Q_{i} \times P_{i+1} \times \ldots \times P_d$. We put $R_0=P$ and $R_d=Q$. Let $t_i$ be the rectangle transposition that permutes $R_{i-1}$ with $R_i$ for every $1 \leq i \leq d$. Then $\tau_{P,Q}=t_1\ldots t_{d-1}t_dt_{d-1}\ldots t_1$ and by the special case above, we know that $t_i$ is a product of two restricted shuffles in direction $i$. Then $s$ is a finite product of restricted shuffles.
\end{proof}

We now consider another special case: that of an element of $\Rec_d$ mapping grid to grid by translating pieces. Beware (see Remark \ref{Remark example of an element which does not send every grid into another grid}) that not every element of $\Rec_d$ has this form.

\begin{Prop}\label{Proposition Special case where we map a grid to another grid}
Every element $f \in \Rec_d$ such that there exists a setwise $\QQ$-free grid-pattern $\cQ$ such that $f(\cQ)$ is a grid-pattern can be written as a finite product of restricted shuffles.
\end{Prop}

\begin{proof}
Let $\cQ=\cQ_1 \times \ldots \times \cQ_d$ and $f(\cQ)=\cQ'_1\times \ldots \times \cQ'_d$, where $\cQ_i$ and $\cQ'_i$ is a partition into intervals of $\interfo{0}{1}$. Thanks to the setwise $\QQ$-freeness of $\cQ$ we know that $f(\cQ)$ is setwise $\QQ$-free, also for every $1 \leq i \leq d$ and every $a \in \interfo{0}{1}$ we have:
\[
\Card(\lbrace I \in \cQ_i \mid \lambda(I)=a \rbrace)=\Card(\lbrace I \in \cQ'_i \mid \lambda(I)=a \rbrace)
\]
Hence there exists an element $g$ of $\IET^d$ such that $g(f(\cQ))=\cQ$. By Proposition \ref{Proposition every element of IET^d is a product of restricted shuffles} we know that $g$ is a finite product of restricted shuffles. Also as $g \circ f$ send $\cQ$ on itself we deduce that $g \circ f$ is a permutation on every maximal subset of translation-isometric rectangles of $\cQ$. Hence it is a product of rectangle transpositions and by Proposition \ref{Proposition any transposition is a product of restricted shuffles} we deduce that $f$ is a finite product of restricted shuffles.
\end{proof}

\begin{Rem}\label{Remark example of an element which does not send every grid into another grid}
For an element of $\Rec_d$ there does not always exist an associated grid-pattern that is sent to another grid-pattern. For example this does not exist in the case of a restricted shuffle $\sigma_{R,s,i}$ of infinite order such that $R \neq \interfo{0}{1}^{d-1}$.
\end{Rem}

\section{Generation by restricted shuffles: bulk of the proof}\label{Section Generation by restricted shuffles}

We now prove Theorem \ref{Theorem the set of all restricted shuffles is a generating subset of REC}, which states that $\Rec_d$ is generated by restricted shuffles. The proof is by induction on the dimension $d$ and the case of the dimension $1$ is already known to be true (Proposition \ref{IETrr}).

Let $d \geq 2$ be the ambient dimension and assume Theorem \ref{Theorem the set of all restricted shuffles is a generating subset of REC} true for $\Rec_{d-1}$.
Let $f \in \Rec_d$ and $\cQ$ be a grid-pattern associated with $f$. Thanks to Proposition \ref{Proposition We can always refine a grid-pattern into a setwise Q-free grid-pattern} we can assume that $\cQ$ is a setwise $\QQ$-free grid-pattern.

We will think of the $d$-th dimension as the ``vertical'' dimension and others as ``horizontal'' dimensions. For every illustration in dimension $2$ we use the element $f_{\mathrm{test}}$ of $\Rec_2$ defined in Figure \ref{fig_test}. The partition $\cP_{\mathrm{test}}$ (on the left of the picture) is associated with $f_{\mathrm{test}}$, and is understood to be setwise $\QQ$-free. We denote by $\cP'_{\mathrm{test}}=f_{\mathrm{test}}(\cP_{\mathrm{test}})$ (on the right of the picture).

\begin{figure}[!ht]
\includegraphics[width=1\textwidth]{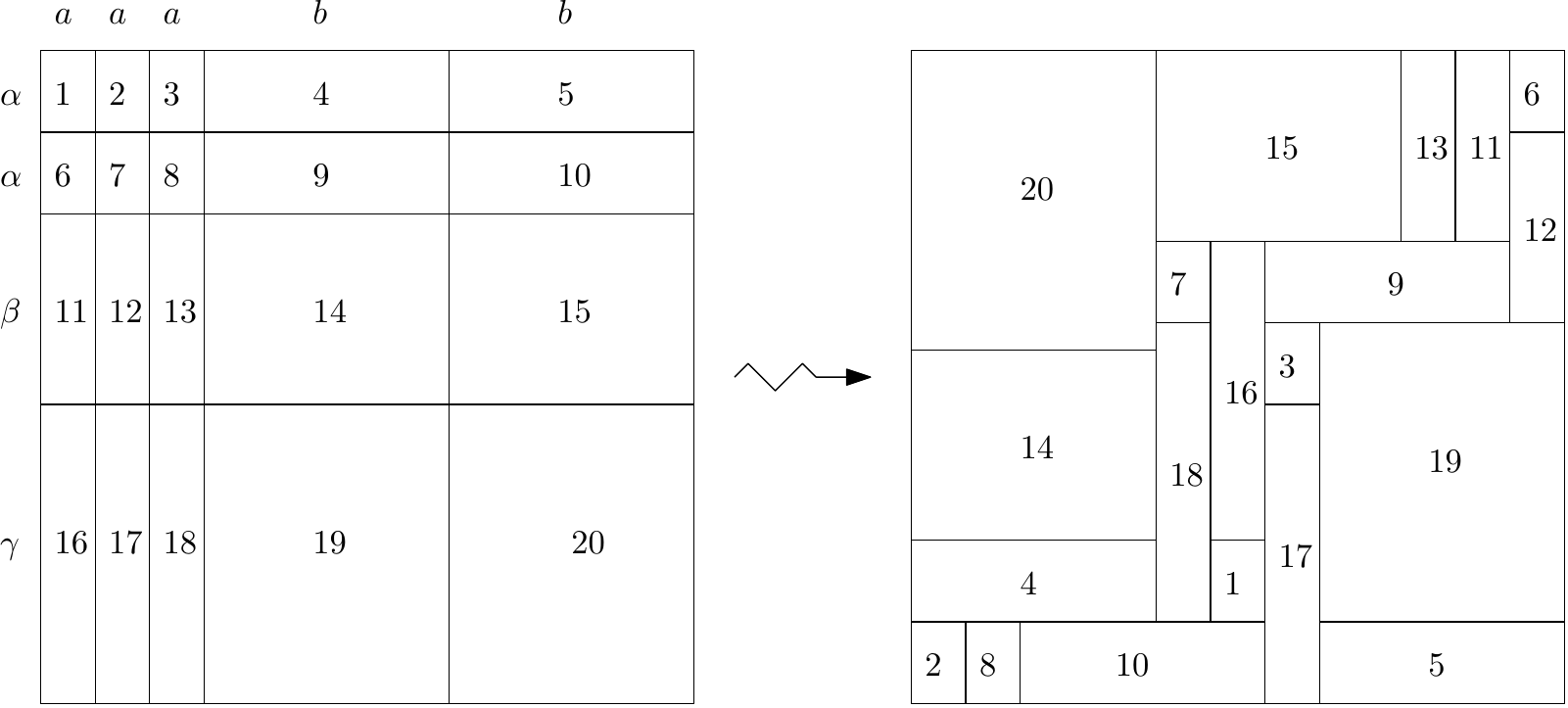}
\caption{Definition of $f_{\mathrm{test}}, \cP_{\mathrm{test}}$ and $\cP'_{\mathrm{test}}$.}
\label{fig_test}\end{figure}

We now introduce a number of simple definitions in this setting, which for this test example is illustrated in the next figures.

\begin{Def}\label{Definition tower above a piece}
Let $\cP$ be a setwise $\QQ$-free rectangle partition of $\interfo{0}{1}^d$. The {\bf ground} of $\cP$ is the following subset of $\cP$:
\[
\Grd(\cP)=\lbrace K \in \cP \mid 0 \in \pr_d(K) \rbrace .
\]
Let $K_0$ be an element of $\Grd(\cP)$. A {\bf tower} above $K_0$ is a subset $T$ of $\cP$ such that:

\begin{enumerate}
\item $K_0 \in T$;
\item $\forall K \in T, ~\pr_d^{\bot}(K)=\pr_d^{\bot}(K_0)$;
\item The set $\bigcup\limits_{K \in T} \pr_d(K)$ is a subinterval of $\interfo{0}{1}$.
\end{enumerate}
The element $K$ of $T$ which satisfies $\sup(\pr_d(K))=\sup \big( \bigcup\limits_{K \in T} \pr_d(K) \big)$ is called the {\bf top} of the tower $T$, denoted by $\Top(T)$.
The {\bf highest tower} above $K_0$, denoted by $T(K_0)$, is the maximal tower above $K_0$ according to the inclusion order.
\end{Def}

\begin{Def}
A {\bf city} of $\cP$ is a subset of $\cP$ containing $\Grd(\cP)$, and which is a union of towers. The {\bf highest city} of $\cP$, denoted by $\City(\cP)$, is the union of all highest towers above elements of the ground $\Grd(\cP)$.
The {\bf top} of a city $\mathcal{V} \subset \cP$ (see Figure \ref{topwork}) is the set of $\Top(T)$ when $T$ ranges over maximal towers in $\mathcal{V}$.
The {\bf sky} of $\cP$, denoted by $\Sky(\cP)$, is the complement of $\City(\cP)$ in $\cP$.

\end{Def}

\begin{figure}[!ht]
\begin{center}
\includegraphics[width=0.5\textwidth]{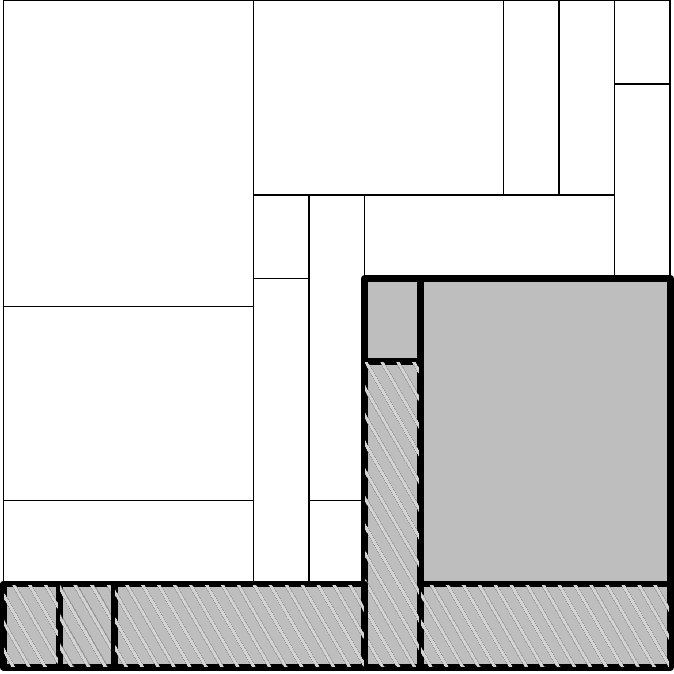}
\caption{Hatched pieces compose the Ground of $\cP'_{\mathrm{test}}$, it is also a city of $\cP'_{\mathrm{test}}$. All grey pieces (hatched or not) compose $\City(\cP'_{\mathrm{test}})$. Full white pieces represent the sky of $\cP'_{\mathrm{test}}$.}\label{Figure ground of the example}
\end{center}
\end{figure}

\begin{Def}
The complexity of $\cP$ is the following subset of $\mathopen{]} 0, 1 \mathclose{[}$:
\[
\mathscr{C}(\cP)= \lbrace \min(\pr_d(K)) \mid K \in \Sky(\cP) \rbrace .
\]
The set $\mathscr{C}(\cP)$ is empty if and only if $\cP= \City(\cP)$. Otherwise, the minimum of the set $\mathscr{C}(\cP)$ is called the {\bf working height} of $\cP$ denoted by $\whei(\cP)$.
\end{Def}

The idea is to move pieces of $\City(\cP)$ with horizontal restricted shuffles so that the new partition $\cP'$ obtained satisfies $\mathscr{C}(\cP') \subset \mathscr{C}(\cP) \smallsetminus \lbrace \whei(\cP) \rbrace$. For this we describe more precisely how and where we move pieces.

\begin{Def}
We define the {\bf building worksite} of $\cP$, denoted by $\Work^-(\cP)$, as the following subset of $\Top(\City(\cP))$:
\[
\Work^-(\cP)=\lbrace K \in \Top(\City(\cP)) \mid \sup(\pr_d(K))=\whei(\cP) \rbrace .
\]
Similarly we define the {\bf upper building worksite} of $\cP$, denoted by $\Work^+(\cP)$ (see Figure \ref{topwork}), as the following subset of $\Sky(\cP)$:
\[
\Work^+(\cP)=\lbrace P \in \Sky(\cP) \mid \min(\pr_d(P))=\whei(\cP) \rbrace .
\]
We define the {\bf site} of $\cP$ (see Figure \ref{figsite}) as the subset of $e_d^{\bot}$ define as the following:
\[
\Site(\cP)=\bigcup\limits_{ K \in \Work^-(\cP)} \pr_d^{\bot}(K) .
\]
\end{Def}

\begin{figure}[!ht]
\begin{center}
\includegraphics[width=0.6\textwidth]{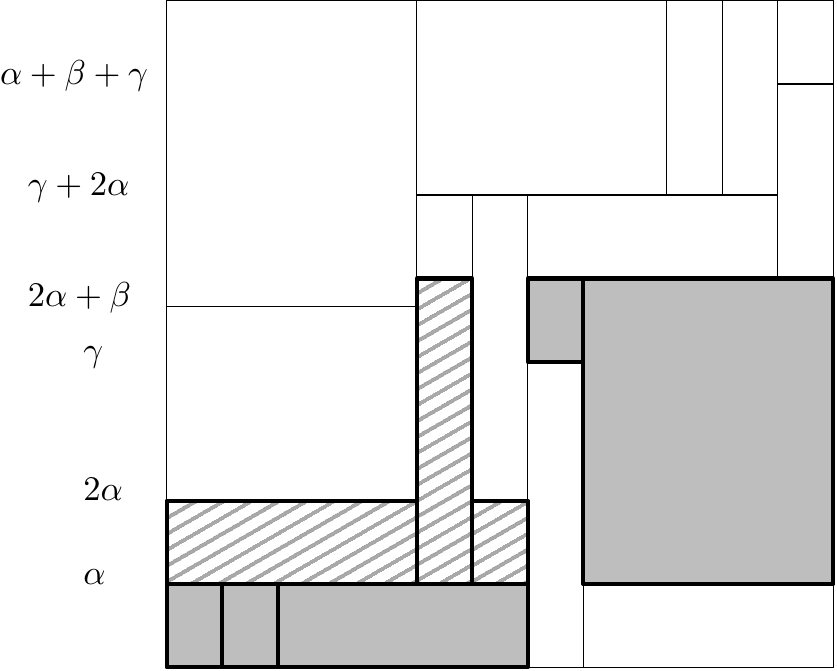}
\caption{The set of all grey pieces represents $\Top(\City(\cP'_{\mathrm{test}}))$ and the set of all hatched pieces represents $\Work^+(\cP)$.}\label{topwork}
\end{center}
\end{figure}

\begin{figure}[!ht]
\begin{center}
\includegraphics[width=0.7\textwidth]{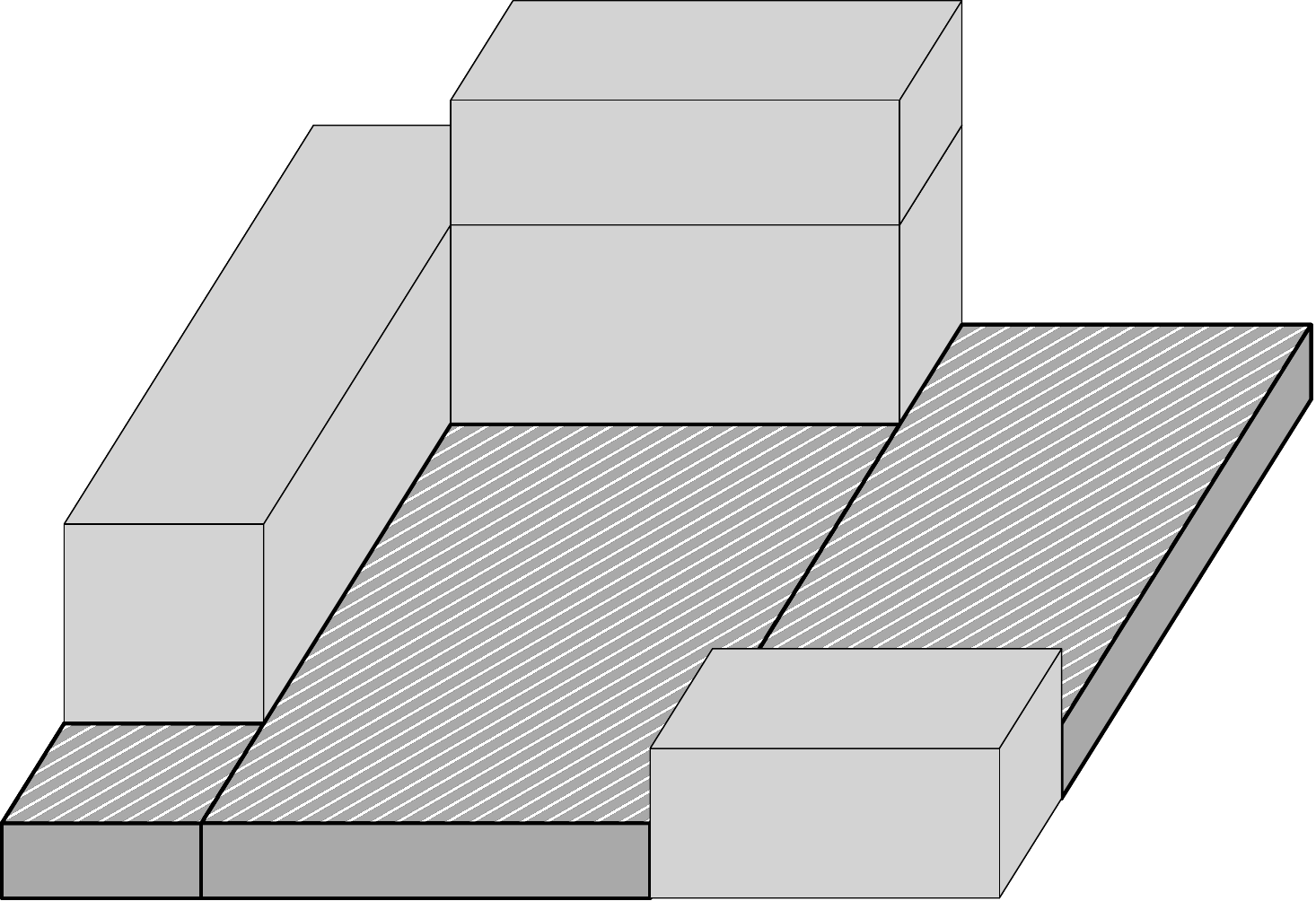}
\caption{In dimension $3$, illustration of a city of a partition where the hatched space represents the site of the partition.}\label{figsite}
\end{center}
\end{figure}

The proof of Theorem \ref{Theorem the set of all restricted shuffles is a generating subset of REC} is done by induction on the cardinal $c$ of $\mathscr{C}(\cP)$. The case $c=0$ is treated in the following lemma.

\begin{Lem}\label{Lemma case of an empty complexity}
Let $\cP$ be a setwise $\QQ$-free partition such that $\mathscr{C}(\cP)=\emptyset$. Then there exists a product $r$ of vertical restricted shuffles such that $\cP$ is associated with $r$ and $r(\cP)$ is a grid-pattern. (See Figure \ref{Figure illustration when the sky is empty}.)
\end{Lem}

\begin{proof}
A consequence of $\whei(\cP) = \emptyset$ is that $\City(\cP)=\cP$, that is, highest towers form a partition of $\mathcal{P}$. In particular, we have a partition $D$ of $e_d^{\bot}$ such that for every $x \in \interfo{0}{1}$ we have $\lbrace \pr_d^{\bot}(K) \mid K \in \cP, ~\textup{and}~ x \in \pr_d(P) \rbrace= D$.  Also the set $\lbrace \pr_d(K) \mid K \in \cP \rbrace$ is setwise $\QQ$-free, thus for every $a \in \interfo{0}{1}$, the number of rectangles $K$ such that $\lambda(\pr_d(K))= a$ is the same in every tower $T \subset \City(\cP)$. Then, using Proposition \ref{IETrr} in each tower, by using only restricted shuffles in direction $d$, we can move pieces inside the tower $T \subset \City(\cP)$ to reorder them according to the length of their projection on $\Vect(e_d)$. The image of $\cP$ by the product of these restricted shuffles is a grid-pattern.
\end{proof}

We now consider the induction step for $c>0$.

\begin{Lem}\label{Lemma diminution of the cardinal of the complexity}
Let $\cP$ be a setwise $\QQ$-free partition such that $\mathscr{C}(\cP) \neq \emptyset$. There exists a product $g$ of horizontal restricted shuffles (i.e., in direction $\neq d$) such that $\cP \in \Pi_g$ and:
\[
\mathscr{C}(g(\cP)) \subset \mathscr{C}(\cP) \smallsetminus \lbrace \whei(\cP) \rbrace .
\]
\end{Lem}

\begin{proof}
For every $1 \leq i \leq d$ define $F_i= \lbrace \lambda(\pr_i(K)) \mid K \in \cP \rbrace$; it is a setwise $\QQ$-free subset of $\RR^+$. Define $\Omega^-=\lbrace \pr_d^{\bot}(K) \mid K \in \Top(\City)(\cP) \rbrace$ and $\Omega^+=\lbrace \pr_2(K) \mid K \in \Work^+ \cup \Top(\City(\cP)) \smallsetminus \Work^- \rbrace$. By definition, $\Omega^-$ and $\Omega^+$ are two partitions of $\interfo{0}{1}^{d-1}$ such that for every $K \in \Omega^- \cup \Omega^+$ and every $1 \leq i \leq d$ we have $\lambda (\pr_i(K)) \in F_i$. Then, by Lemma \ref{Lemma two partitions with Q-conditions contains the same number of pieces} we deduce that there exists $\delta \in \Rec_{d-1}$ such that $\Omega^- \in \Pi_{\delta}$ (for every element $K$ of $\Omega^-$, the restriction of $\delta$ to $K$ is a translation) and $\delta(\Omega^-)=\Omega^+$ and for every $K \in \Omega^- \cap \Omega^+$ we have $\delta(K)=K$. As we assumed Theorem \ref{Theorem the set of all restricted shuffles is a generating subset of REC} in dimension $d-1$, we know that $\delta$ can be written as the product of restricted shuffles of $\Rec_{d-1}$. Then we define $g \in \Rec_d$ such that:
\[
g(x)=\left\{
\begin{array}{cc}
(\delta \times \Id) (x) & \textup{if}~ \pr_2(x) < \whei(\cP)\\
x & \textup{else}.

\end{array}
\right.
\]

From this definition we obtain that $g$ is the product of restricted shuffles in $\Rec_d$ with direction in $\lbrace 1,2 \ldots, d-1 \rbrace$. Also by definition of $\delta$ we obtain that for every $K \in \Grd(\cP)$ we have $g(T(K)) \subset T(g(K))$ and $g(\Sky(\cP))=\Sky(\cP)$. This implies $\mathscr{C}(g(\cP)) \subset \mathscr{C}(\cP)$. Also as $\delta(\Omega^-)=\Omega^+$ we deduce that for every $K \in \Sky(\cP)$ such that $\min(\pr_d(K))=\whei(\cP)$ there exists $Q_K \in \Grd(\cP)$ such that $\delta(\pr_d^{\bot}(Q_K)=\pr_d^{\bot}(K)$. Hence we have $K \in T(g(Q_K))$ and this implies that $\whei(\cP) \notin \mathscr{C}(g(\cP))$.
\end{proof}

Then by induction on the cardinal of the complexity we deduce the following proposition:

\begin{Prop}\label{Proposition reduction to a partition with an empty complexity}
Let $\cQ$ be a setwise $\QQ$-free grid-pattern of $\interfo{0}{1}^d$. For every $f \in \Rec_d$ such that $\cQ \in \Pi_f$, there exists a finite product $r_f$ of restricted shuffles such that $f(\cQ) \in \Pi_{r_f}$ and $\mathscr{C}(r_f(f(\cQ)))=\emptyset$.
\end{Prop}

Thanks to Proposition \ref{Proposition reduction to a partition with an empty complexity} and Proposition \ref{Lemma case of an empty complexity} we deduce Theorem \ref{Theorem the set of all restricted shuffles is a generating subset of REC}.

The mains steps in the proof are illustrated in Figure \ref{Figure induction to reduce the cardinal of the complexity}.

\begin{figure}[!ht]
\begin{center}
\includegraphics[width=0.8\textwidth]{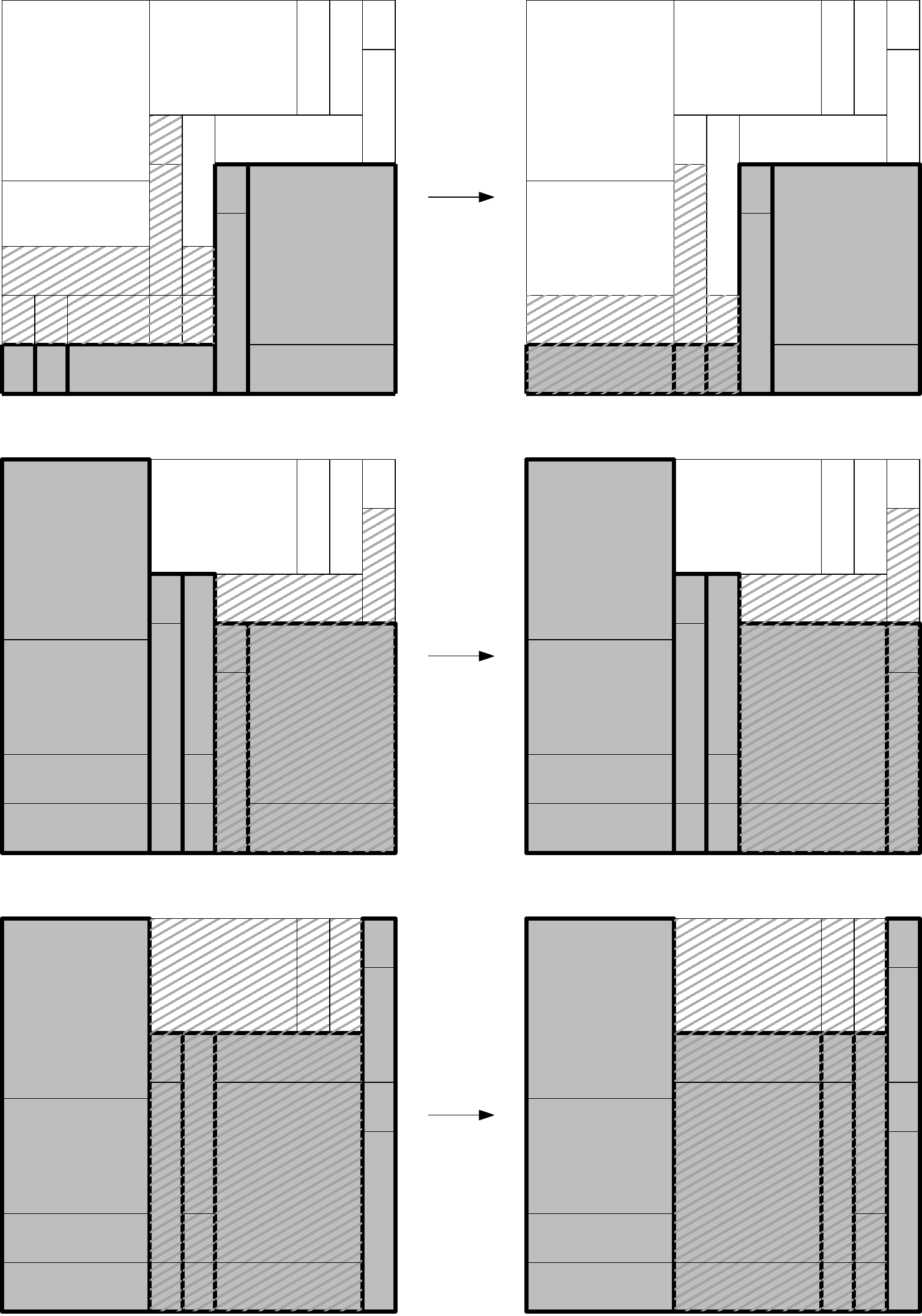}
\caption{Iterations to find a product $r$ of restricted shuffles such that $\cP'_{\mathrm{test}} \in \Pi_r$ and $\Sky(r(\cP'_{\mathrm{test}}))=\emptyset$. On each left picture, all grey pieces represent the highest city, all grey hatched pieces represent towers whose top's height is the complexity of the partition and all white hatched pieces represent pieces of sky of the partition which are also in the upper work.}\label{Figure induction to reduce the cardinal of the complexity}
\end{center}
\end{figure}

\newpage

\begin{figure}[!ht]
\begin{center}
\includegraphics[width=0.9\textwidth]{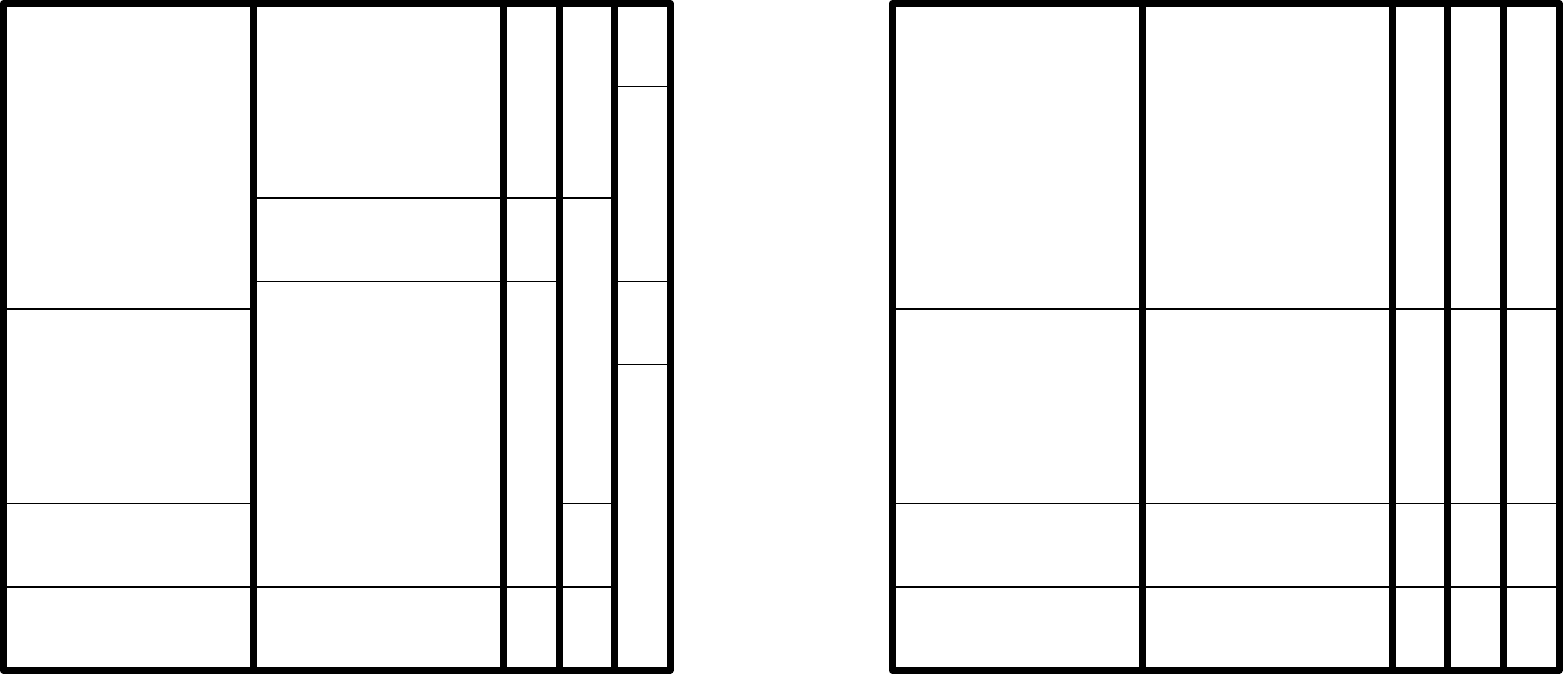}
\caption{Illustration of what looks like a setwise $\QQ$-free partition with an empty sky and how moving pieces inside each tower can lead to a setwise $\QQ$-free grid-pattern.}\label{Figure illustration when the sky is empty}
\end{center}
\end{figure}


\section{Generation by restricted shuffles: a refinement for $\Rec_2$}\label{Section A refinement for REC 2}

Here we establish a more precise and concrete statement in dimension 2. Theorem \ref{Theorem the set of all restricted shuffles is a generating subset of REC} says every element $f$ in $\Rec_d$ can be obtained as a composition of restricted shuffles. It is tempting to improve this statement by fixing a setwise $\QQ$-free partition $\cP \in \Pi_f$, and then shuffling rectangles in $f(\cP)$ without changing the partition. The proof seems at first sight to provide this, but the induction step forces to change the partition. In dimension 2, we can avoid this, see Theorem \ref{Theorem reordering of a setwise Q-free partition can lead to a grid-pattern} below.

In this case we can be more precise than Theorem \ref{Theorem the set of all restricted shuffles is a generating subset of REC}.

\begin{Def}
Let $\cP$ be a partition into rectangles of $\interfo{0}{1}^d$. A {\bf restricted shuffle on $\cP$} is a restricted shuffle which shuffles two rectangles of $\cP$.
For $n \in \mathbf{N}^*$, a {\bf $n$-sequence of restricted shuffles on $\cP$} is a sequence $(r_1,\ldots,r_n)$ of restricted shuffles such that for every $1 \leq i \leq n$ the element $r_i$ is a restricted shuffle on $r_{i-1} \circ \ldots \circ r_1(\cP)$. The partition $r_n \circ \ldots \circ r_1 (\cP)$ is called the image of $\cP$ by this sequence.
\end{Def}

Here is the refined version of Theorem \ref{Theorem the set of all restricted shuffles is a generating subset of REC}, in dimension 2

\begin{Thm}\label{Theorem reordering of a setwise Q-free partition can lead to a grid-pattern}
Suppose $d=2$. For every $f \in \Rec_d$ and for every setwise $\QQ$-free partition $\cP \in \Pi_f$, there exists a sequence of restricted shuffles $(r_1,\ldots,r_n)$ on $\cP$ such that $f=r_n \circ \ldots \circ r_1$.
\end{Thm}

\begin{Rem}
To motivate the setwise $\QQ$-free property, we illustrate with a partition which is in the image by $\Rec_d$ of a grid-pattern $\cQ$, and which is not setwise $\QQ$-free. Indeed if we do not allow to cut pieces of $\cQ$ then for every sequence of restricted shuffles on $\cQ$, the image of $\cQ$ by this sequence is always $\cQ$.

\begin{figure}[!ht]
\begin{center}
\includegraphics[width=0.7\textwidth]{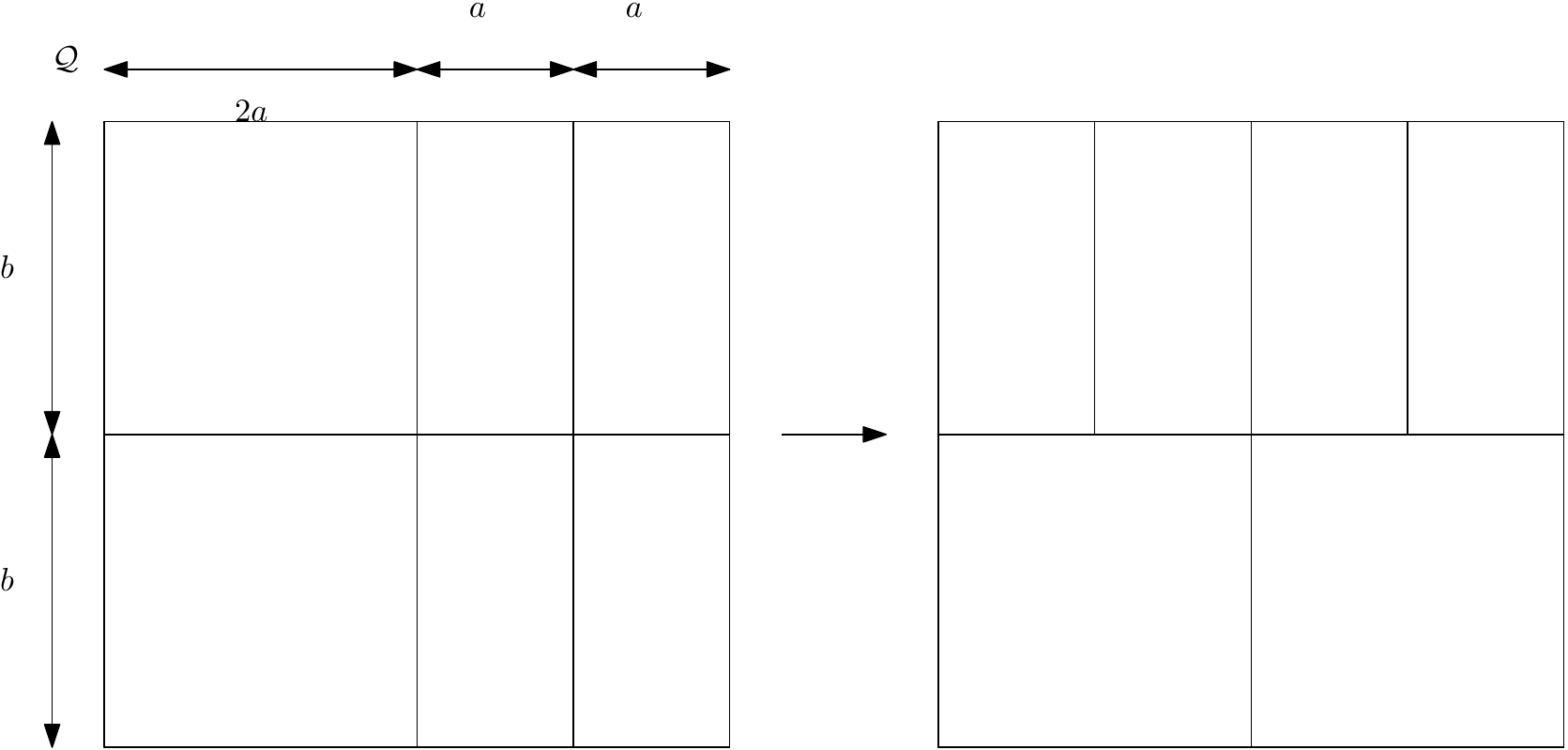}
\caption{\textbf{Left:} A grid-pattern $\cQ$ which is not setwise $\QQ$-free. \textbf{Right:} A rearrangement of $\cQ$ which is not the image of $\cQ$ by a sequence of restricted shuffles on $\cQ$.}
\end{center}
\end{figure}
\end{Rem}

Theorem \ref{Theorem reordering of a setwise Q-free partition can lead to a grid-pattern} in dimension $1$ is Proposition \ref{IETrr}. In dimension $2$, we begin with two refinements of Proposition \ref{Proposition Special case where we map a grid to another grid} and Lemma \ref{Lemma case of an empty complexity} obtained with immediate changes.

\begin{Lem}
Let $d \in \NN^*$, let $f \in \Rec_d$ such that there exists a setwise $\QQ$-free grid-pattern $\cQ$ such that $f(\cQ)$ is a grid-pattern. Then there exists a sequence $(r_1,\ldots,r_n)$ of restricted shuffles on $\cQ$ such that the image of $\cQ$ by this sequence is $f(\cQ)$, in particular we have $f=r_n \circ \ldots \circ r_1$. \qed
\end{Lem}

\begin{Lem}
Let $\cP$ be a setwise $\QQ$-free partition such that $\mathscr{C}(\cP)=\emptyset$. Then there exists a sequence of restricted shuffles on $\cP$ such that the image of $\cP$ by this sequence is a grid-pattern. \qed
\end{Lem}

With these two results, the proof of Theorem \ref{Theorem reordering of a setwise Q-free partition can lead to a grid-pattern} is the same as the one of Theorem \ref{Theorem the set of all restricted shuffles is a generating subset of REC} until Lemma \ref{Lemma diminution of the cardinal of the complexity}, where we proved the following refinement in dimension $2$:

\begin{Lem}
Suppose $d=2$. Let $\cP$ be a setwise $\QQ$-free partition. There exists a product $g$ of restricted shuffles in direction inside $\lbrace 1,2, \ldots, d-1 \rbrace$ such that $\cP \in \Pi_g$ and:
\[
\mathscr{C}(g(\cP)) \subset \mathscr{C}(\cP) \smallsetminus \lbrace \whei(\cP) \rbrace .
\]
Also there exists a sequence $(r_1,\ldots,r_n)$ of restricted shuffles on $\cP$ such that $g=r_n \circ \ldots \circ r_1$.
\end{Lem}

\begin{proof}
First we rearrange every tower of $\City(\cP)$ such that pieces of every tower is ordered by increasing order about their length of their $2$-projection.

We recall that $\Omega^-= \lbrace \pr_d^{\bot}(K) \mid K \in \Top(\City(\cP)) \rbrace$ and $\Omega^+= \lbrace \pr_d^{\bot}(K) \mid K \in \Work^+ \cup \Top(\City(\cP)) \smallsetminus \Work^- \rbrace$. by Lemma \ref{Lemma two partitions with Q-conditions contains the same number of pieces} we deduce that there exists $\delta \in \Rec_{d-1}$ such that $\Omega^- \in \Pi_{\delta}$ (for every element $K$ of $\Omega^-$, the restriction of $\delta$ to $K$ is a translation) and $\delta(\Omega^-)=\Omega^+$ and for every $K \in \Omega^- \cap \Omega^+$ we have $\delta(K)=K$.
The main argument is that every connected component $C$ of $\Site(\cP)=\bigsqcup\limits_{K \in \Omega^-} K$ is a left half-open interval and there exists $\Omega_C^- \subset \Omega^-$ which partitions $C$. Similarly we can define the subset $\Omega_C^+$ of $\Omega^+$ which partitions $C$. Then by $\QQ$-freeness we can also ask $\delta$ to send $\Omega_C^-$ on $\Omega_C^+$. Then we define $g_C \in \Rec_2$ such that:
\[
g_C(x)= \left\{
\begin{array}{cc}
(\delta \times \Id)(x) & \textup{if}~ \pr_2(x) < \whei(\cP) ~\textup{and}~ \pr_1(x) \in C \\
x & \textup{else}
\end{array}
\right.
\]
We can see that $g_C$ only moves towers of $\City(\cP)$. And as these towers are rearrange such that pieces of every tower is ordered by increasing order about their length of their $2$-projection. We deduce that there exists a sequence $(r_1,\ldots ,r_n)$ of restricted shuffle on $\cP$ such that $g_C=r_n \circ \ldots \circ r_1$.
Let $g$ be the product of every $g_C$ where $C$ ranges over the set of all connected components of $\Site(\cP)$. It satisfies the statement of the lemma.
\end{proof}

At this point we are unable to prove Theorem \ref{Theorem reordering of a setwise Q-free partition can lead to a grid-pattern} for arbitrary d. Here are some possible step towards a proof.

\begin{Def}
Define $(S_{\delta})$ as the following statement.
For every $R$ be finite union of rectangles in $\interfo{0}{1}^{\delta}$. Let $\cP,\cQ$ be rectangle partitions of $R$. Suppose that for each i there is a $\QQ$-free subset $F_i$ of $\mathopen{]}0,1\mathclose{[}$ such that for every $K \in \cP \cup \cQ$, we have $\lambda(\pr_i(K)) \in F_i$. Then one can change $\cQ$ into $\cP$ by a finite sequence of shuffles.
\end{Def}

Then the statement $S_{d-1}$ implies Theorem \ref{Theorem reordering of a setwise Q-free partition can lead to a grid-pattern} in dimension $d$, the argument being an immediate adaptation of the above one.

Indeed, we know that $S_1$ holds. Here $R$ is just a disjoint union of intervals, and the difficulty is that components of $R$ can have complicated shapes in general. Note that proving $(S_{\delta})$ immediately reduces to the case when $R$ is connected; however it sounds convenient not to assume $R$ connected in order to set up a proof (e.g., by induction on the number of the rectangles).

\section{Rectangle exchanges in multirectangles}\label{gen_multir}
\subsection{Generation by restricted shuffles}

Fix the dimension $d\ge 1$. Let $R=A\sqcup B$ be a multirectangle in $\RR^d$, with $A,B$ non-empty disjoint multirectangles. Let $\Gamma\{A,B\}$ (resp.\ $\Gamma(A,B)$) be the subgroup of $\Rec(R)$ preserving the partition $\{A,B\}$ (resp.\ preserving $A$ and $B$).

\begin{Lem}\label{stabmax}
The subgroup $\Gamma(\{A,B\})$ is a maximal proper subgroup of $\Rec(R)$. If $A,B$ are $\Rec$-isomorphic then the only proper subgroup strictly containing $\Gamma(A,B)$ is its overgroup of index 2 in $\Gamma\{A,B\}$ (while if they are not $\Rec$-isomorphic, then $\Gamma(A,B)=\Gamma\{A,B\}$).
\end{Lem}

Note that $\Gamma(A,B)\simeq\Rec(A)\times\Rec(B)$.
\begin{proof}
Let $X\subset A$, $Y\subset B$ be multirectangles and $f$ an isomorphism $A\to B$. Define an involutive element $\sigma_f$ as equal to $f$ on $X$, $f^{-1}$ on $Y$, and identity elsewhere. Let $E$ be the set of all such involutive elements (for all possible $X$, $Y$, $f$). We first claim that $\Rec(R)=\Gamma(A,B)E$.

Fix $u\in\Rec(R)$. Define $A_1=A\cap u^{-1}(A)$, $A_2=A\cap u^{-1}(B)$, $B_1=B\cap u^{-1}(A)$, $B_2=B\cap u^{-1}(B)$. Clearly $A=A_1\sqcup A_2$ and $B=B_1\sqcup B_2$. Also $A=u(A_1)\sqcup u(B_1)$. Hence $\vl_d(A_2)=\vl_d(A)-\vl_d(A_1)=\vl_d(A)-\vl_d(u(A_1))=\vl_d(u(B_1))$. Hence, by Lemma \ref{vl_rec}, there is a $\Rec$-isomorphism $\xi:A_2\to u(B_1)$. Similarly, there is a $\Rec$-isomorphism $B_1\to u(A_2)$, which we still denote by $\xi$ (since $A_2$ and $B_1$ are disjoint this is harmless). Then define $w$ as equal to $u$ on $A_1\sqcup B_2$ and to $\xi$ on $A_2\cup B_1$. Then $w$ is bijective, hence belongs to $\Gamma(A,B)$. Define $v=w^{-1}\circ u$. Then $v(A_1)=A_1$, $v(B_2)=B_2$, and $v$ exchanges $A_2$ and $B_1$. Hence $v\in E$. So $u=wv\in \Gamma(A,B)E$.

Let us now improve the claim. Say that a rectangle $W$ in $A$ is small if there exists a translate of $W$ in $A$ disjoint from $W$ and sharing a face with $W$ perpendicular to the first coordinate; similarly define a small rectangle in $B$. Let $E'$ be the set of elements $f$ of $E$ exchanging one nonempty small rectangle $X$ of $A$ and a small rectangle $Y$ of $B$ through a translation (thus note $f=\psi_{X,Y}$).

The second claim is that for every $\psi\in E'$ we have $\Rec(R)=\langle \Gamma(A,B),\psi\}$. By the first claim, it is enough to prove that every $\sigma\in E$, we have $\sigma\in\langle \Gamma(A,B),\sigma\}$. Decomposing $\sigma$ into a product with disjoint support, we can suppose that $\sigma\in E'$, say $\sigma=\psi_{X',Y'}$. Cut the rectangle $X'$ along the first direction into two isomorphic rectangles $X'_1$, $X'_2$. After conjugating by an element of $\Gamma(A,B)$ (and possibly exchanging the names of $X'_1$ and $X'_2$), we can suppose that $X'_1\subset X$ and $X'_2\cap X=\emptyset$. Similarly, conjugating by an element of $\Gamma(A,B)$, we can suppose that $Y'_1\subset Y$ and $Y'_2\cap Y=\emptyset$. Define $q=\psi_{X'_1,X'_2}\psi_{Y'_1,Y'_2}\in \Gamma(A,B)$. Then $\sigma=(q\psi q^{-1})\psi\in\langle \Gamma(A,B),\psi\rangle$ and the second claim is proved.

Now, to prove the lemma, we have to prove that $\langle \Gamma(A,B), f\rangle=\Rec(R)$ for every $f\notin \Gamma\{A,B\}$. Indeed, up to switch $A$ and $B$, we have $A\cap f^{-1}(A)$ and $A\cap f^{-1}(B)$ non-empty. Choose nonempty rectangles $U,V$ in these two subsets, translate of each other, on each of which $f$ is a translation. Choosing $U,V$ small enough, we can ensure that $f(U)$ and $f(V)$ are small in $A$ and $B$ respectively. Define $\psi_{U,V}\in \Gamma(A,B)$. Then $f\tau f^{-1}=\psi_{f(U),f(V)}\in\langle \Gamma(A,B),f\rangle$. By the second claim, we deduce that $\langle \Gamma(A,B),f\rangle=\Rec(R)$.\end{proof}

\begin{Rem}
Keeping Corollary \ref{derin} in mind, the above proof also shows the same statement in restriction to the derived subgroup: if $f\in D(\Rec(R))\smallsetminus \Gamma\{A,B\}$ then $\langle D(\Rec(A)),D(\Rec(B)),f\rangle=D(\Rec(R))$.
\end{Rem}

\begin{Coro}\label{gen_two_multi}
For every multirectangle $M$ written as union $M=M_1\cup M_2$ of two non-disjoint multirectangles $M_1,M_2$, we have $\langle\Rec(M_1),\Rec(M_2)\rangle=\Rec(M)$.
\end{Coro}
\begin{proof}
Let $H$ be the subgroup generated by $\Rec(M_1)\cup\Rec(M_2)$. Write $M_3=M_2\smallsetminus M_1$. If $M_1$ or $M_3$ is empty the conclusion is trivial, hence suppose otherwise. So $M$ is the disjoint union $M_1\sqcup M_3$. If by contradiction $H\neq\Rec(M)$, by Lemma \ref{stabmax} we have $H\subset\Gamma\{M_1,M_3\}$. Let $C$ be a rectangle strictly contained in $M_1\cap M_2$ with a translate $C'$ strictly contained in $M_3$, and let $f$ be the rectangle transposition exchanging $C$ and $C'$. Then $f\in\Rec(M_2)$, so $f\in H$. But clearly $f\notin\Gamma\{M_1,M_3\}$. We obtain a contradiction.
\end{proof}

\begin{Coro}\label{gen_shuffles_multirec}
For every multirectangle $M$ with connected interior, $\Rec(M)$ is generated by restricted shuffles. 
\end{Coro}
\begin{proof}
We first assume that we can write $M$ as a finite disjoint union of rectangles $R_1\sqcup\dots \sqcup R_n$, where for each $k\ge 2$, the multirectangle $(R_1\sqcup\dots\sqcup R_{k-1})\cap R_k$ is non-empty. Then the result follows from Corollary \ref{gen_two_multi} (and the case $n=1$) by an immediate induction on $n$.

To show that we can write $M$ in this way, we can write $M$ as a finite disjoint union of rectangles $R_1\sqcup\dots \sqcup R_n$, where for each $k$, $R_1\sqcup\dots\sqcup R_k$ has a connected interior. For each $k$ with $2\le k\le n$, let $C_k\subset R$ be a rectangle such that $C_k$ meets both $R'=R_1\cup\dots R_{k-1}$ and $R_k$ (see Figure \ref{indrk}).
\begin{figure}[h]
\setlength{\unitlength}{6mm} 
\centering
\begin{picture}(10,5) 
 \put(0,0){\line(0,1){2}}
 \put(0,2){\line(1,0){2}}
 \put(0,0){\line(1,0){10}}
 \put(2,1){\line(1,0){5}}
 \put(2,1){\line(0,1){1}}
 \put(7,1){\line(0,1){.9}}
 \put(7,2.1){\line(0,1){.8}}
 \put(7,3.1){\line(0,1){.9}}
 \put(7,3.6){\line(1,0){2}}
 \put(9,2){\line(0,1){1.6}}
 \put(9,2){\line(1,0){1}}
 \put(10,0){\line(0,1){2}}
 \put(3,1){\line(0,1){3}}
 \put(3,4){\line(1,0){4}}
 \put(5,2){\line(1,0){3}}
 \put(5,2){\line(0,1){1}}
 \put(5,3){\line(1,0){3}}
 \put(8,2){\line(0,1){1}}
 \put(.7,.8) {$R'$}
 \put(3.3,3) {$R_k$}
 \put(5.3,2.3) {$C_k$}
\end{picture}
\caption{$R'$, $R_k$ and $C_k$} 
\label{indrk} 
\end{figure}
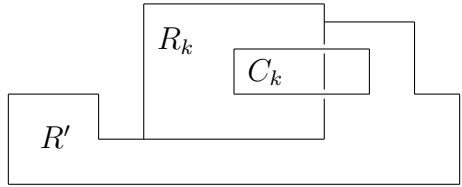
Then $R=R_1\cup C_2\cup R_2\cup\dots C_n\cup R_n$ is a description satisfying the previous requirement.
\end{proof}

\begin{Rem}
The connectedness assumption is necessary: in general let $U_1,\dots,U_n$ the connected components of the interior of $R$. Each $U_i$ is the interior of some multirectangle $R_i$, and $R$ is the disjoint union of the $R_i$; then every restricted shuffle preserves each $U_i$. Actually it follows from Corollary \ref{gen_shuffles_multirec} that the subgroup generated by restricted shuffles equals $\prod_i\Rec(R_i)$, the component-wise stabilizer of the decomposition $R=\bigsqcup_{i=1}^nR_i$.
\end{Rem}

\subsection{Rectangle exchange transformations in tori}\label{rec_torus}

Let $\Lambda$ be a lattice in $\RR^d$. We can define a rectangle in the torus $\RR^d/\Lambda$ as the image of a rectangle in $\RR^d$, and define accordingly $\Rec_d(\RR^d/\Lambda)$.

\begin{Prop}
For every lattice $\Lambda$ in $\RR^d$ there exists a multirectangle $M$ that is a fundamental domain for $\Lambda$, in the sense that $\RR^d$ is the disjoint union of all $M+\omega$ for $\omega$ ranging over $\Lambda$.
\end{Prop}
\begin{proof}
Let $R$ be a rectangle such that $R+\Lambda=\RR^d$. Choose an invariant total ordering on $\Lambda$ and write $\Lambda_+=\{\omega\in\Lambda:\omega>0\}$. Define $M=\{x\in R:\forall\omega\in\Lambda_+:x-\omega\notin R\}$. We claim that $M$ is a multirectangle, and is a fundamental domain. First observe that $M=R\smallsetminus \left(\bigcup_{\omega\in\Lambda_+} R\smallsetminus (R+\omega)\right)$. And indeed since $R$ is bounded and $\Lambda$ acts properly on $\RR^d$, this union is a finite union. Hence $M$ is a multirectangle. If $\omega\in\Lambda_+$ then $M\cap (M+\omega)$ is empty. Since the order is total, it follows that the $\Lambda$-translates of $M$ are pairwise disjoint. Finally, for $x\in\RR^d$, the set of $\omega\in\Lambda$ such that $x+\omega\in R$ is nonempty and finite. Let $\omega$ be its maximum. Then we see that $x+\omega\in M$. Hence $M+\Lambda=\RR^d$.
\end{proof}

Hence, in a sense, the $\Rec_d$ of tori are particular cases of $\Rec_d$ of multirectangles. However, it can be useful to see them as tori. For instance, if we choose $\Lambda$ such that the group of diagonal matrices preserving $\Lambda$ is infinite (this occurs for some $\Lambda$, but not for $\Lambda=\ZZ^2$), we obtain somewhat exotic automorphisms of $\Rec_d(\RR^d/\Lambda)$.

\begin{Rem}For any lattice $\Lambda$, one can define $\vl_d(\RR^d/\Lambda)$ as $\vl_d(M)$ for some fundamental domain $M$ as above: this does not depend on $M$.
For the lattice $\Lambda$ with basis $((a,b),(c,d))$ (for real numbers $a,b,c,d$, with $ad-bc>0$) we have observed experimentally that $\vl_2(\RR^2/\Lambda)$ equals $a\otimes d-c\otimes b$.
\end{Rem}

\section{The derived subgroup}\label{Section The derived subgroup}

Fix a nonempty multirectangle $M$ in $\RR^d$. Let $\Transpo_d(M)$ be the subset of all rectangle transpositions in $\Rec_d(M)$. In this section, we prove that $\Transpo_d(M)$ is a generating subset of $D(\Rec_d(M))$, and that the latter is a simple group.

We start with some preliminary observations.

\begin{Lem}\label{Lemma every element of order 2 is a product of transpositions}
Every element of order 2 in $\Rec_d(M)$ is a product of rectangle transpositions with pairwise disjoint support.
\end{Lem}

\begin{proof}
Let $f\in \Rec_d$ be an element of order 2. For $v \in \RR^d$, define $X_v=\lbrace x:f(x)-x=v \rbrace$. Note that $X_v \cup X_{-v}$ is $f$-invariant. Choose a subset $V_+$ of $\RR^d$ of elements called ``positive elements'', such that $\RR^d$ is the disjoint union $V_+\sqcup -V_+ \sqcup \lbrace 0 \rbrace$. For $v$ positive, choose a finite partition $\mathcal{W}_v$ of $X_v$ into rectangles, and let $\mathcal{W}$ be the union, for $v$ positive, of all $\mathcal{W}_v$. Then $f$ is the (disjoint support) product of all $\tau_{K,f(K)}$ for $K$ ranging over $\mathcal{W}$.
\end{proof}

\begin{Lem}
\begin{enumerate}[(a)]
\item\label{drec_a} $\Transpo_d(M) \subset D(\Rec_d(M))$.
\item\label{drec_b} If $D(\Rec_d)$ is simple then it is generated by $\Transpo_d$.
\end{enumerate}
\end{Lem}

\begin{proof}
(\ref{drec_a}) Let $f \in \Transpo_d$ and $P$ and $R$ be the two rectangles switched by $f$. We can decompose $P=P_1\sqcup P_2$ such that $P_1$ and $P_2$ are translation-isometric. Let $f_1$ be the element that switches $P_1$ with $f(P_1)$ and let $f_2$ be the element that switches $P_1$ with $P_2$ and $f(P_1)$ with $f(P_2)$. Then we have $f=[f_1,f_2]$.

(\ref{drec_b}) From Lemma \ref{Lemma every element of order 2 is a product of transpositions}, it follows that the subgroup $N$ generated by $\Transpo_d(M)$ coincides with the subgroup generated by elements of order $2$. By (a), $N\subset D(\Rec_d)$. Hence, if $D(\Rec_d)$ is simple, it follows that $N=D(\Rec_d)$.
\end{proof}

For $d=1$, simplicity of $D(\Rec_d)=D(\IET)$ was proved by Sah \cite{Sah} and it follows that $D(\Rec_1)$ is generated by rectangle transpositions. Vorobets \cite{Vorobets2011} more recently reproved simplicity of $D(\Rec_1)$, by first proving that it is generated by transpositions. Our approach for arbitrary $d \geq 1$ is inspired by the latter.

\begin{Def}
For every $\eps>0$ we define $\Transpo_d^{\eps}(M)$ as the set of all rectangle transpositions $\tau_{K,L}$ such that each of $K$, $L$ is contained in a square of length $\eps$ contained in $M$.
\end{Def}

\begin{Prop}\label{Proposition Generating subsets of D(REC)}
\begin{enumerate}[(a)]
\item\label{gendrec_a} The subset $\Transpo_d(M)$ generates $D(\Rec_d(M))$.
\item\label{gendrec_b} For every $\eps >0$, the subset $\Transpo_d^{\eps}(M)$ generates $D(\Rec_d(M))$.
\item\label{gendrec_c} For every nonempty multirectangle $U\subset M$, the group $D(\Rec_d(M))$ is normally generated by $\Transpo_d(U)$.
\end{enumerate}
\end{Prop}

\begin{proof}
(\ref{gendrec_a}) First suppose that $M$ has connected interior. So by Corollary \ref{gen_shuffles_multirec}, $\Rec_d(M)$ is generated by restricted shuffles. From usual commutator formulas it follows that in a group, every commutator $[a_1 a_2 \dots , b_1 b_2 \dots]$ is a product of conjugates of the $[a_i,b_j]$.
We deduce that every commutator of elements in $\Rec_d$ can be written as the product of conjugates of commutators of restricted shuffles. Hence thanks to Lemma \ref{Lemma every element of order 2 is a product of transpositions} we deduce that it is enough to prove that every commutator of restricted shuffles is a product of elements of order $2$. We already saw that this statement is true in dimension $1$. Let $i,j \in \lbrace 1,2 \ldots, d \rbrace$ and $s,s'$ be two restricted rotations and $R,R'$ be two $(d-1)$-subrectangles of $\interfo{0}{1}^{d-1}$. We have different cases:
\begin{enumerate}[(1)]
\item If $i=j$ then for every $x \in \interfo{0}{1}^d$ and for every $k \in \lbrace 1 ,\ldots, d \rbrace$ with $k \neq i$, we have $[\sigma_{R,s,i},\sigma_{R',s',i}](x)_k=x_k$. Also $\pr_i^{\bot}(x) \notin R \cap R'$ we have $[\sigma_{R,s,i},\sigma_{R',s',i}](x_i)=x_i$ and if $\pr_i^{\bot}(x) \in R \cap R'$ we have $[\sigma_{R,s,i},\sigma_{R',s',i}](x)_i=[s,s'](x_i)$. Then by using the result in dimension $1$ we deduce that $[\sigma_{R,s,i},\sigma_{R',s',i}]$ is a product of elements of order $2$.
\item Let assume $i \neq j$. We remark that if $R=R_1 \sqcup R_2$ then $\sigma_{R,s,i}=\sigma_{R_1,s,i} \circ \sigma_{R_2,s,i}$. Then by using again the equality between commutators we deduce that it is enough to show that the commutator $[\sigma_{R,s,i},\sigma_{R',s',i}]$ is a product of elements of order $2$, where $R$ and $R'$ are as small as we want. In particular as $i \neq j$ we can assume that $R$ and $R'$ are small enough such that for every $x \in \support(\sigma_{R,s,i}) \cap \support(\sigma_{R',s',i})$ we have both $\sigma_{R,s,i}(x) \notin \support(\sigma_{R',s',i})$ and $\sigma_{R',s',i}(x) \notin \support(\sigma_{R,s,i})$. Then in this case the commutator $[\sigma_{R,s,i},\sigma_{R',s',i}]$ permutes cyclically three disjoint rectangles by translations. Hence it is a product of two rectangle transpositions.
\end{enumerate}
In general, there exists a rectangle isomorphism of $M$ onto a multirectangle $M'$ with connected interior. So $\Rec_d(M)\simeq\Rec_d(M')$, and by the previous case, $\Rec_d(M')$ and hence $\Rec_d(M)$ is generated by elements of order 2. By Lemma \ref{Lemma every element of order 2 is a product of transpositions}, we deduce that $\Rec_d(M)$ is generated by $\Transpo_d(M)$.

(\ref{gendrec_b}) This immediately follows from (\ref{gendrec_a}) by writing a rectangle transposition as product of rectangle transpositions with pairwise disjoint, small enough support.

(\ref{gendrec_c}) Let $U$ contain a cube of length $\varepsilon$. The group $D(\Rec_d(M))$ is, by (\ref{gendrec_b}), generated by $\Transpo_d^{\eps/2}(M)$. Each element of $\Transpo_d^{\eps/2}(M)$ is conjugate to an element of $\Transpo_d(U)$. Whence the result.
\end{proof}

We deduce the simplicity of the derived subgroup $D(\Rec_d(M))$:

\begin{Thm}\label{Theorem Simplicity of D(REC)}Let $M$ be a nonempty multirectangle in $\RR^d$. 
Every nontrivial subgroup of $\Rec_d$ normalized by $D(\Rec_d(M))$ contains $D(\Rec_d(M))$. In particular:
\begin{enumerate}
\item[a)] The group $D(\Rec_d(M))$ is simple.
\item[b)] The group $D(\Rec_d(M))$ is contained in every nontrivial normal subgroup of $\Rec_d(M)$.
\end{enumerate}
\end{Thm}

\begin{proof}
Let $N$ be a nontrivial subgroup of $\Rec_d(M)$ normalized by $D(\Rec_d(M))$. Let $f$ be a non-identity element of $N$. For some $\eps$, there exists a square $K$ of length $\eps$ contained in $M$, such that $f$ is a translation on $K$ and such that $K$ and $f(K)$ are disjoint.

Let us prove that every rectangle transposition $\tau_{P,Q}$ with $P \cup Q \subset K$ belongs to $N$. By Proposition \ref{Proposition Generating subsets of D(REC)} (c) this yields the conclusion.

Cut $P$ and $Q$ in two equal halves according to the $d$-coordinate: let $P_1$ and $Q_1$ be their lower halves, and $P_2$, $Q_2$ their upper halves. Then $[f,\tau_{P_1,Q_1}]$ permutes $P_1$ and $Q_1$ by translations, permutes $f(P_1)$ and $f(Q_1)$ by translations, and is identity elsewhere. Let $s$ permute $P_2$ and $f(P_1)$ by translations, $Q_2$ and $f(Q_1)$ by translations, and be identity elsewhere. Then $s[f,\tau_{P_1,Q_1}]s^{-1}=\tau_{P,Q}$. Hence $\tau_{P,Q}\in N$.
\end{proof}

Thus the group $\Rec_d$ is {\bf monolithic}, in the sense that the intersection of all nontrivial normal subgroups is nontrivial.

\section{Abelianization of $\Rec_d$}\label{Section Abelianization of REC}

\subsection{The case $d=1$ revisited}
For expositional purposes, it is convenient to reprove the case $d=1$ and then write down the necessary elaboration.
The SAF homomorphism is defined as follows: for $f\in\IET$, define 
\[S(f)=\sum_{x\in\RR}\lambda\big((f-\mathrm{id})^{-1}(\{x\})\big)\otimes x\in\RR^{\otimes 2}.\]

By a direct verification, $S$ is a group homomorphism, called SAF homomorphism (or scissors congruence homomorphism). For a restricted rotation of size $b$ on an interval of size $a$, it takes the value $a\otimes b-b\otimes a$. Hence the image contains $\Lambda^2_\QQ\RR$, and since restricted rotations generate $\IET$, the image is equal to $\Lambda^2_\QQ\RR$ (which we identify to the kernel of the canonical projection $\RR^{\otimes 2}\to S^2\RR$). The SAF homomorphism factors through a surjective group homomorphism $\bar{S}:\IET_\ab\to\Lambda^2_\QQ\RR$. This was independently observed by Sah and Arnoux--Fathi. Sah \cite{Sah} then proved that $\bar{S}$ is injective, that is, this is precisely the abelianization homomorphism (the proof was then reproduced in \cite{Arnoux}), that is, the inclusion $D(\IET)\subseteq\Ker(S)$ is an equality; we now reprove this.

We need the following purely group-theoretic lemma, which is the key algebraic fact, and is not explicit in the original proof given in~\cite{Arnoux}. It will be used when we deal with arbitrary $d\ge 1$.

\begin{Lem}\label{uab}
Consider the (additive) abelian group $V$ with presentation: generated by the $u_{a,b}$, $0<b<a\le 1$, subject to the relators (whenever they make sense):
\begin{enumerate}
\item\label{ax1} $u_{a+a',b}=u_{a,b}+u_{a',b}$;
\item\label{ax2} $u_{a,b+b'}=u_{a,b}+u_{a,b'}$;
\item\label{ax3} $2u_{a,b}+u_{2b,a}=0$.
\setcounter{saveenum}{\value{enumi}}
\end{enumerate}
\noindent (``Whenever\dots'' means $0<b<\min(a,a')$, $a+a'\le 1$ in (\ref{ax1}), $0<\min(b,b')$, $b+b'<a\le 1$ in (\ref{ax2}), and $0<b<a<2b\le 1$ in (\ref{ax3}).)

Then the assignment $u_{a,b}\mapsto a\wedge b$ induces a group isomorphism from $V$ to $\bigwedge^2_\QQ\RR$.
\end{Lem}

Fix $k\ge 0$. For $w,w',w''\in\RR^k$, we say that $w''$ is a 1-coordinate sum of $w$ and $w'$ if there exists $i$ such that $w_j=w'_j=w''_j$ for all $j\neq i$ and $w''_i=w_i+w'_i$ (hence denoting $\bar{w}=w_1\otimes\dots w_k$, we have $\bar{w}''=\bar{w}+\bar{w}'$). Note that the 1-coordinate sum of $w,w'$ is not always defined and not always unique.
Lemma \ref{uab} is the particular case (for $k=2$) of (\ref{isowedge}) of the next lemma.
 
\begin{Lem}\label{pres_tensor}
Consider the abelian group $W_k$ with generators $v_w$, $w\in\mathopen]0,1]^k$ and relators $v_{w''}=v_w+v_{w'}$ for all $w,w',w''\in\mathopen]0,1\mathclose[^k$ such that $w''$ is a 1-coordinate sum of $w$ and $w'$. Then 

\begin{enumerate}
\item\label{isoten} The group homomorphism $f:W_k\to\RR^{\otimes k}$ mapping $v_w$ to $\bar{w}$, is a group isomorphism.

\item\label{isotenab} If, for $k\ge 2$, we define $W'_k$ by the same presentation, but only considering those generators $v_w$ for which $w_{k-1}>w_k$, then the resulting canonical map $W'_k\to\RR^{\otimes k}$ is also an isomorphism.

\item\label{isowedge} For $k\ge 2$, starting from the presentation defining $W'_k$, define $W''_k$ by modding out by the elements of the form $2v_{w_1,\dots,w_{k-2},a,b}+v_{w_1,\dots,w_{k-2},2b,a}$ for $w_1,\dots,w_{k-2},a,b\in\mathopen]0,1]$ such that $b<a<2b$. Then the resulting canonical map $W'_k\to\RR^{\otimes (k-2)}\otimes\bigwedge^2\RR$ is an isomorphism.
\end{enumerate}
\end{Lem} 

\begin{proof}
Start with (\ref{isoten}). The case $k=1$, which underlies the general case, is very standard and left to the reader.

We prove only $k=2$ as the case $d\ge 3$ is strictly similar. We write $v_{x,y}$ rather than $v_{(x,y)}$. Define, for arbitrary $y\in\RR^*$ and $x\in \mathopen]0,1\mathclose[$, $v_{x,y}$ as $nv_{x,y/n}$ where $n$ is the number of smallest absolute value such that $0<y/n<1$, and define $v_{x,0}=0$. Applying the case $d=1$ for fixed $x\in \mathopen]0,1\mathclose[$, we see that $y\mapsto v_{x,y}$ is an injective group homomorphism. We now do the same for fixed $y\in\RR$ and thus define $v_{x,y}$ for arbitrary $x,y\in\RR^2$, so that $(x,y)\mapsto v_{x,y}$ is $\ZZ$-bilinear. Hence it induces a surjective group homomorphism $v:\RR^{\otimes 2}\to W_k$ (where the tensor product is over $\ZZ$, or equivalently over $\QQ$). We have $f\circ v=\mathrm{id}_{\RR^{\otimes 2}}$. Since $v$ is surjective, this implies that $f$ is an isomorphism.

(We used that the canonical homomorphism $\RR^{\otimes k}_\ZZ\to\RR^{\otimes k}_\QQ$ is an isomorphism. This holds because $\RR^{\otimes k}_\ZZ$ is a torsion-free divisible group. In turn, this holds because for every $n\ge 1$, multiplication by $n$ is invertible, namely with inverse $(x_1\otimes\dots\otimes x_k)\mapsto ((x_1/n)\otimes x_2\otimes\dots\otimes x_k)$.)

For (\ref{isotenab}), we also suppose $k=2$ to simplify the notation, the proof in general being the same. In $W'_k$, for $a,b$ with $0<a<b\le 1/2$ we have $2v_{a/2,b}=v_{a,b}$. It follows that for arbitrary $a,b\in \mathopen]0,1]$, the element $2^nv_{2^{-n}a,b}$ is well-defined for $n$ large enough, and independent of $n$, and moreover equals $v_{a,b}$ when $a<b$. We therefore denote it $v_{a,b}$ as well. Then the $v_{a,b}$ satisfy the same additivity relators without the restriction $a<b$. Indeed, for $n$ large enough,
 \[v_{a+a',b}=2^nv_{2^{-n}(a+a'),b}=2^nv_{2^{-n}a,b}+2^nv_{2^{-n}a',b}=v_{a,b}+v_{a',b}\]
and
\[v_{a,b+b'}=2^nv(2^{-n}a,b+b')=2^nv(2^{-n}a,b)+2^nv(2^{-n}a,b')=v_{a,b}+v_{a,b'}.\]

For (\ref{isowedge}), using the isomorphism of (\ref{isotenab}), the additional relators correspond to modding out by the elements $2(c\otimes a\otimes b+c\otimes b\otimes a)$ for $0<b<a<2a\le 1$ and $c$ of the form $w_1\otimes\dots w_{k-2}$ with $0<w_i\le 1$. An arbitrary element $c\otimes a\otimes b+c\otimes b\otimes a$ ($a,b\in\RR$, $c$ of the same form with $w_i$ arbitrary real numbers) is a $\ZZ$-linear combination of such elements. Hence the given map  $W'_k\to\RR^{\otimes (k-2)}\otimes\bigwedge^2\RR$ defines an isomorphism $V\to\bigwedge^2\RR$.
\end{proof}

\begin{Prop}\label{rab}
For $0\le b\le a\le 1$, let $R_{a,b}\in\IET$ be the restricted rotation ``$+b$ modulo $a$'' on $[0,a\mathclose[$, identity elsewhere. Explicitly, it is given by $x\mapsto x+b$ on $[0,a-b\mathclose[$, $x\mapsto x-a+b$ in $[a-b,a\mathclose[$, and identity on $[a,1\mathclose[$. Then in the abelianization of $\IET$, they satisfy all relators of Lemma \ref{uab}.
\end{Prop}
\begin{proof}
We write multiplicatively. Relator (\ref{ax2}) is clear, as the equality even holds in $\IET$.

For the relator (\ref{ax1}), first consider the conjugate $R^{[a]}_{a',b}$ of $R_{a',b}$ by $R_{a+a',a}$: it is thus identity outside $[a,a+a'\mathclose[$; it acts as $x\mapsto x+b$ on $[a,a+a'-b\mathclose[$, and $x\mapsto x+b-a'$ on $[a+a'-b,a+a'\mathclose[$.  
A direct computation shows that $R^{[a]}_{a',b}\circ R_{a,b}\circ R_{a+a',b}^{-1}$ equals the ``transposition'' that permutes by translations the disjoint intervals $[0,b\mathclose[$ and $[a,a+b\mathclose[$; this is a commutator. Hence (\ref{ax1}) holds in the abelianization.

For (\ref{ax3}), we compute that $R_{a,b}^2=R_{a,2b-a}$, and then $R_{a,b}^2\circ R_{2b,a}$ is the ``transposition'' that permutes by translations the disjoint intervals $[0,2b-a\mathclose[$ and $[a,2b\mathclose[$. Hence this is a commutator.
 \end{proof}

By Lemma \ref{uab} and Proposition \ref{rab}, there is a well-defined group homomorphism $F:\Lambda^2_\QQ\RR\to\IET_{\ab}$ such that for all $0<b<a\le 1$, $F(a\wedge b)=\pi(R_{a,b})$, where $\pi$ is the projection $\IET\to\IET_{\ab}$.

\begin{Lem}\label{lemsec}
$F\circ S=\pi$. In particular, $F\circ\bar{S}=\mathrm{id}$, and thus $\bar{S}$ is a group isomorphism.
\end{Lem}
\begin{proof}
Let $H$ be the subgroup of $\IET$ consisting of those $f$ such that $F\circ S(f)=\pi(f)$. This is a subgroup of $\IET$ containing the derived subgroup, and hence is a normal subgroup. Hence, since $\IET$ is normally generated by the $R_{a,b}$, it is enough to check that $R_{a,b}\in H$. Indeed, $F(S(R_{a,b}))=F(a\wedge b)=\pi(R_{a,b})$. 
\end{proof}

\subsection{The general case $d\ge 1$: generalized SAF homomorphism}\label{gsaf}

The generalized SAF homomorphism was briefly described in the introduction. To describe its image, it is convenient to perform a simple change of variables. 
Let $\sigma_i$ be the linear automorphism of $\RR^{\otimes d}$ transposing the $i$-th and $d$-th coordinates. Define $\vl_{d,i}=\sigma_i\circ\vl_d$, where $\vl_d$ is the tensor volume (Section \ref{s_tensor_volume}). Thus
\[\vl_{d,i}(I_1\times\dots\times I_d)=\lambda(I_1)\otimes\dots\otimes\lambda(I_{i-1})\otimes\lambda(I_{i+1})\otimes
\dots\otimes\lambda(I_d)\otimes\lambda(I_i)\]
for all left-closed right-open bounded intervals $I_1,\dots,I_d$.

For $f\in\Rec_d$, define 
\begin{align*}
T_i(f) = & \sum_{x\in\RR^d}\vl_{d,i}\big((f-\mathrm{id})^{-1}(\{x\})\big)\otimes x_i\\
       = & \sum_{\alpha\in\RR}\vl_{d,i}\big((f-\mathrm{id})_i^{-1}(\{\alpha\})\big)\otimes \alpha\qquad \in\RR^{\otimes (d+1)},\end{align*}
and $T(f)=(T_1(f),\dots,T_d(f))\in(\RR^{\otimes (d+1)})^d$, and call it generalized SAF homomorphism.

By a computation similar to the 1-dimensional one (using that the ``measure'' $\vl_{d,i}$ is invariant under elements of $\Rec_d$), we obtain that $T_i$ is a group homomorphism, and hence so is $T$.

Fix $i\in\{1,\dots,d\}$. For $0<b<a\le 1$ and $c\in\mathopen]0,1]^{d-1}$, first define $c^{\sharp(i,a)}=(c_1,\dots,c_{i-1},a,c_{i},\dots,c_{d-1})\in\RR^d$ (beware of the shift of coordinates). Define $R_{i,c,a,b}$ as being identity outside $K_{c,i}^{a}=\prod_{j=1}^d[0,c^{\sharp(i,a)}_j\mathclose[$, and shuffling by $b$ on the $i$-coordinate inside $K_{c,i}^a$: $R_{a,b}$ on the $i$-coordinate and identity on other coordinates. More explicitly, it is given by translation by $be_i$ on $K_{c,i}^{a-b}$ and translation by $(b-a)e_i$ on $(a-b)e_i+K_{c,i}^b$.

Then $T_j(R_{i,c,a,b})=0$ for $j\neq i$, while $T_i(R_{i,c,a,b})=c\otimes (a\otimes b-b\otimes a)$.

 Since the $R_{i,c,a,b}$ generate $\Rec_d$ as a normal subgroup (as consequence of Theorem \ref{rec_shu_gen}), it follows that the image of $T$ is exactly $\big(\RR^{\otimes (d-1)}\otimes\bigwedge^2\RR\big)^d$. 
 
It remains to prove that the inclusion $D(\Rec_d)\subseteq\Ker(T)$ is an equality. From Proposition \ref{rab}, when $i$ is fixed as well as $c$, the elements $R_{i,c,a,b}$ satisfy the relators of Lemma \ref{uab}. We need a simple elaboration of that lemma, when $c$ is allowed to vary, namely Lemma \ref{ucab} below.

The following is essentially a restatement of Lemma \ref{pres_tensor}(\ref{isowedge}).

\begin{Lem}\label{ucab}
Consider the (additive) abelian group $V_k$ with presentation: generated by the $u_{w,a,b}$, $0<b<a\le 1$, $w\in\mathopen]0,1]^k$, subject to the relators of Lemma \ref{uab} for fixed $w$ [that is, whenever meaningful, (\ref{ax1}) $u_{w,a+a',b}=u_{w,a,b}+u_{c,a',b}$, (\ref{ax2}) $u_{w,a,b+b'}=u_{w,a,b}+u_{c,a,b'}$, (\ref{ax3}) $2u_{w,a,b}+u_{w,2b,a}=0$], and the additional relators:

\begin{enumerate}
 \setcounter{enumi}{\value{saveenum}}
\item\label{ax5} $u_{w'',a,b}=u_{w,a,b}+u_{w',a,b}$ if $w''$ is a 1-coordinate sum of $w,w'$, whenever it makes sense $0<b<a\le 1$.
\end{enumerate}

Then the homomorphism mapping $u_{w,a,b}\mapsto \bar{w}\otimes (a\wedge b)$ from $V$ to $\RR^{\otimes k}\otimes\bigwedge^2_\QQ\RR$ is a group isomorphism.\qed
\end{Lem}

Therefore, there is a group homomorphism $M_i:\RR^{\otimes (d-1)}\otimes\bigwedge^2\RR\to(\Rec_d)_\ab$ such that for all $c\in\mathopen]0,1]^d$ and all $0<b<a\le 1$ we have $M_i(c\otimes (a\wedge b))=\pi(R_{i,c,a,b})$. By $\ZZ$-linearity this defines a homomorphism $M:(\RR^{\otimes(d-1)}\otimes\bigwedge^2\RR)^d\to(\Rec_d)_\ab$ by $M(\eta_1,\dots,\eta_d)=\sum_{i=1}^dM_i(\eta_i)$.

Here is the analogue of Lemma \ref{lemsec}.
\begin{Lem}\label{lemsecd}
$M\circ T=\pi$. In particular, $M\circ\bar{T}=\mathrm{id}$, and thus $\bar{T}$ is a group isomorphism.
\end{Lem}
\begin{proof}
Let $H$ be the subgroup of $\Rec_d$ consisting of those $f$ such that $M\circ T(f)=\pi(f)$. This is a subgroup of $\Rec_d$ containing the derived subgroup, and hence is a normal subgroup. Hence, since $\Rec_d$ is normally generated by the $R_{i,c,a,b}$, it is enough to check that $R_{i,c,a,b}\in H$. Indeed, 
\begin{align*}
M(T(R_{i,c,a,b}))=& M\Big(\sum_j T_j(R_{i,c,a,b})\Big)\\
=M(T_i(R_{i,c,a,b})) &= M_i(c\otimes (a\wedge b))=
\pi(R_{i,c,a,b}).\qedhere\end{align*} 
\end{proof}

\begin{Coro}
Let $G$ be the subgroup of $\Rec_2$ generated by those restricted shuffles that consist of shuffling a square inside a rectangle (e.g., those elements $R_{b,a,b}$). Then $G$ is a proper subgroup, containing the derived subgroup.
\end{Coro}
\begin{proof}
We have $T_1(R_{1,b,a,b})=(b\otimes (a\wedge b),0)$ (and similarly for $T_2$) and hence all ``square'' restricted shuffles have an image of this form. If $(a_i)_{i\in I}$ is a $\QQ$-basis of $\RR$ and we fix a total order on $I$, we see that $T_1(G)$ has the basis $a_i\otimes (a_i\wedge a_j)$ for $i\neq j$, and $a_{i}\otimes (a_j\wedge a_k) - a_k\otimes (a_i\wedge a_j)$, $a_j\otimes (a_k\wedge a_i)- a_k\otimes (a_i\wedge a_j)$ for $i<j<k$. (And $T(G)=T_1(G)\times T_2(G)$.) In particular, $a_i\otimes (a_j\wedge a_k)$ is not in the image. In particular, whenever $(a,b,c)$ is $\QQ$-free with $0<b<a\le 1$, $0<c\le 1$, we have $R_{1,c,a,b}\notin G$.
\end{proof}
 
\subsection{A normal subgroup larger than the derived subgroup}\label{Section a normal subgroup bigger than the derived subgroup}~

We denote by $\GtG_d$ the subgroup of $\Rec_d$ generated by $\IET^d \cup \Transpo_d$ (where $\IET^d$ acts component-wise).

\begin{Coro}\label{Proposition GtG is a normal subgroup of REC}
The group $\GtG_d$ is a normal subgroup of $\Rec_d$ and containing $D(\Rec_d)$; for $d\ge 2$ both inclusions $D(\Rec_d)\subset\GtG_d\subset\Rec_d$ are strict.
\end{Coro}

\begin{proof} For $0\le b\le a$, denote by $R_{a,b}$ the restricted rotation (as defined in Proposition \ref{rab}), and $R_{a,b}^{(i)}$ (for $1\le i\le d$) the element of $\IET^d$ acting as $R_{a,b}$ on the $i$th component and identity on other components. Then $T_i(R^{(i)}_{a,b})=1^{\otimes (d-1)}\otimes (a\wedge b)$ and $T_j(R^{(i)}_{a,b})=0$ for $j\neq i$. Hence $T(\GtG_d)$ contains the subgroup $(\{1^{\otimes (d-1)}\}\otimes_\QQ \bigwedge_\QQ^2 \RR)^d$, which is not trivial: this already shows that the inclusion $D(\Rec_d)\subset\GtG_d$ is proper. In addition, since $T$ vanishes on $\mathcal{T}_d$ and since the $R_{a,b}^{(i)}$ (for varying $a,b,i$) normally generate $\IET^d$, it follows that this inclusion is an equality. For $d\ge 2$, $\{1^{\otimes (d-1)}\}\otimes \bigwedge^2 \RR$ is a proper subgroup of $\RR^{\otimes (d-1)}\otimes \bigwedge^2 \RR$ (all tensor products being over $\QQ$) and it follows that $\GtG_d$ is a proper subgroup.
\end{proof}

\begin{Rem}
The notation $\GtG_d$ is for ``Grid-to-Grid''. Let $S$ be the subset of $\Rec_d$ consisting of elements $f$ such that there exists a grid-pattern associated $\cQ$ such that $f(\cQ)$ is still a grid-pattern. Then $S$ contains $\IET^d \cup \Transpo_d$ but is not equal to $\GtG_d$. However the normal closure in $\Rec_d$ of $S$ is $\GtG_d$.
\end{Rem}

\subsection{Generalized SAF homomorphism on an arbitrary multirectangle}

(Recall that all tensor products are over $\QQ$.)

Let $M\subseteq\RR^d$ be a multirectangle. The definition of generalized SAF-homomorphism in \S\ref{gsaf} works without any change, yielding a homomorphism $T:\Rec_d(M)\to \big(\RR^{\otimes (d-1)}\otimes\bigwedge^2\RR\big)^d$. It is surjective for the same obvious reason.

\begin{Prop}
The kernel of $T:\Rec_d(M)\to \big(\RR^{\otimes (d-1)}\otimes\bigwedge^2\RR\big)^d$ equals the derived subgroup of $\Rec_d(M)$.
\end{Prop}
\begin{proof}
It is enough to prove that the kernel of $T$ is contained in the derived subgroup, the other inclusion being obvious.

It is convenient to consider the whole group $\Rec_d^\square$ as the union over all $M$ of $\Rec_d(M)$. That is, these are compactly supported Rec-automorphisms of $\RR^d$.

We let $\RR^*$ act on $\RR^{\otimes (d+1)}$ by $t\cdot (x_1\otimes\dots\otimes x_{d+1})=(tx_1\otimes \dots \otimes tx_{d+1})$. This induces an action on its subspace $\RR^{\otimes (d-1)}\otimes\bigwedge^2\RR$ and hence on $\big(\RR^{\otimes (d-1)}\otimes\bigwedge^2\RR\big)^d$.

Let $s$ be an affine homothety of $\RR^d$, i.e., an affine automorphism whose linear part is the nonzero scalar multiplication by $\theta_s$. For $f\in\Rec_d^\square$, we readily see that $T(s\circ f\circ s^{-1})=s\cdot T(f)$. Now $M$ being given, fix an affine homothety $s$ such that $s(M)\subset [0,1\mathclose[^d$.

Let $f$ be such that $T(f)=0$. Then $s\circ f\circ s^{-1}\in\Rec_d$, and $T(s\circ f\circ s^{-1})=s\cdot T(f)=0$. 
By the case of $[0,1\mathclose[^d$, we deduce that $s\circ f\circ s^{-1}$ is a product of commutators $[g_i,h_i]$ in $\Rec_d$. Hence $f$ is the product of commutators $[s^{-1}\circ g_i\circ s,s^{-1}\circ h_i\circ s]$ in $\Rec_d(M)$. 
\end{proof}

The fact that the abelianization homomorphism is ``independent'' of $M$ has the following consequence on derived subgroups.

\begin{Coro}\label{derin}
For nonempty multirectangles $M\subseteq M'$ in $\RR^d$, denoting $G=\Rec_d(M')$ and $H=\Rec_d(M)$, we have $HD(G)=G$ and $H\cap D(G)=D(H)$.\qed
\end{Coro}

\section{Rectangle exchanges with flips}\label{sflip}

There is an issue in defining the group of interval exchange with flips, due to the fact that this group does not really act on the interval: this is only an action modulo indeterminacy on finite subsets, and that it cannot be realized as an action on the interval is proved in \cite{CorMod}.

Define a small subset in $\RR^d$ as a subset that is contained in a finite union of affine hyperplanes (here we could content ourselves with hyperplanes of the form $x_i=c$). Consider the set $E$ of maps $[0,1\mathclose[^d\to [0,1]^d$ that are left-continuous in each variable, such that there is a finite partition of $[0,1\mathclose[^d$ into rectangles such that on each cube, it is given by an affine map whose linear part is diagonal with $\pm 1$ diagonal coefficients. Define $\Rec^{\bowtie}_d$ as the set of elements in $E$ that are injective outside a small subset. If $f,g\in\Rec^{\bowtie}_d$, then $g\circ f$ is defined outside a small subset, and coincides with a unique element of $\Rec^{\bowtie}_d$, which we define as $gf$. This makes $\Rec^{\bowtie}_d$ a group (we omit the routine details), which for $d=1$ is known as group of interval exchanges with flips. 

\begin{Prop}
The group $\Rec^{\bowtie}_d$ is simple.
\end{Prop}
\begin{proof}
Let $N$ be a nontrivial normal subgroup. Let $g$ be a nontrivial element of $\Rec^{\bowtie}_d$. There exists a rectangle $M$ that is mapped by $g$ onto a rectangle disjoint of $M$ by an isometry. Let $h$ be an element of $\Rec_d$ that is a rectangle transposition between two rectangles that are both contained in $M$. Then $ghg^{-1}$ is a rectangle transposition between two rectangles that are contained in $g(M)$. Hence $ghg^{-1}h^{-1}$ is a nontrivial element of $\Rec_d$ of order $2$. Since $\Rec_d$ has a torsion-free abelianization and simple derived subgroup, we deduce that $N$ contains the derived subgroup of $\Rec_d$.

Consider a restricted shuffle. Since every rotation of the circle is a product of two reflections, we can write it as a product of two ``restricted reflections''. By a simple argument, every restricted reflection is conjugate in $\Rec^{\bowtie}_d$ to an element if $\Rec_d$ (necessarily in the derived subgroup). Since restricted reflections generate $\Rec_d$, we deduce that $N$ contains $\Rec_d$.

Every element of $\Rec^{\bowtie}_d$ is obviously the product of an element of $\Rec_d$ and a product of elements with pairwise disjoint support, each of which is supported by a single rectangle and acts as a self-isometry of this rectangle. In turn, such an element can be written as a product of such elements for which the self-isometry is a reflection according to some coordinate reflection. Such elements are ``restricted reflections'' and hence belong to $N$. Hence $N=\Rec^{\bowtie}_d$.
\end{proof}

\begin{proof}[Proof of Corollary \ref{recbowsimple}]
We only sketch the proof, since it is quite standard once Theorem \ref{Theorem D(REC) is generated by Transpo} is granted.

Let $N$ be a nontrivial normal subgroup, and take nontrivial $g\in N$, let $R$ be a rectangle on which $g$ acts as a single isometry, with $g(R)$ and $R$ disjoint. Let $h$ be an element of $\Rec_d$ of order 2, consisting of exchanging two small rectangles contained in $R$. Then $ghg^{-1}$ is also of order 2, exchanging two small rectangles contained in $g(R)$. Since every nontrivial normal subgroup of $\Rec_d$ contains its derived subgroup $\Rec'_d$, we deduce that $\Rec'_d\subseteq N$. 

Now let $r$ be a restricted shuffle in $\Rec_d$, with support $R$. Write it as $r=w^2$, with $w$ a restricted shuffle with support $R$. Let $s$ be the reflection with same support $R$ and switching the same direction. Then $sws^{-1}=w^{-1}$, and hence $wsw^{-1}s^{-1}=w^2=r$. Also, it is not hard to check that $s$ is conjugate to an element in $\Rec_d$ (we omit the simple argument, which is the same as in the case $d=1$). Hence $r=[w,s]$ belongs to $N$. Since restricted shuffles generate $\Rec_d$, we deduce $\Rec_d\subseteq N$.

Finally, every element in $\Rec^{\bowtie}_d$ can be written as $tu$ where $u\in\Rec_d$ and $t$ is a product with disjoint support $\prod_i t_i$, where each $t_i$ is supported by a single rectangle and is an isometry of this rectangle. Hence $t_i$ has order 2 and it is not hard again to check that $t_i$ is conjugate to an element of $\Rec_d$. Hence $N=\Rec_d^{\bowtie}$.
\end{proof}

\begin{Prop}\label{ori_lift}
There exists an injective group homomorphism $\Rec_d^{\bowtie}$ into $\Rec_d$. More precisely, denote by $C'$ the cube $[-1,1\mathclose[^d$. Then the ``centralizer'' in $\Rec_d(C')$ of the 2-elementary abelian group $B_d$ of order $2^d$ consisting of those $(x_1,\dots,x_d)\mapsto (\eps_1x_1,\dots,\eps_dx_d)$, $\eps_i\in\{\pm 1\}$, is naturally isomorphic to $\Rec_d^{\bowtie}([0,1\mathclose[^d)$. Here centralizer means those element which commute with these maps outside a smalle subset (small meaning contained in a finite union of hyperplanes). 
\end{Prop}
\begin{proof}
For $f\in\Rec_d^{\bowtie}$ and $x\in\mathopen]0,1\mathclose[^d)$ at the neighborhood of which $f$ is an isometry, define $d_f(x)$ as the differential of $f$ at $x$ (this is a diagonal matrix with diagonal entries in $\{\pm 1\}$). Then define $q(f)(x)$ as $d_f(x)f(x)$. More generally define $q(f)(Ax)$ as $Ad_f(x)f(x)$ for every diagonal matrix $A$ with diagonal entries in $\{\pm 1\}$. Then $q(f)\in\Rec_d(C')$.

Conversely, for $g\in\Rec_d(C')$ centralizing $B_d$ and $x\in\mathopen]0,1\mathclose[^d$ at the neighborhood of which $g$ is a translation, define $d_g(x)$ as the diagonal matrix with diagonal coefficients in $\{\pm 1\}$, such that the sign of $d_g(x)_{i,i}$ is the same as the sign of $g(x)_i$. Then define $r(g)(x)=d_g(x)g(x)$, and more generally $r(g)(Ax)=Ad_g(x)g(x)$ for any diagonal matrix $A$ with diagonal entries in $\{\pm 1\}$.

Then the reader can check that $r,q$ are group homomorphisms and are inverse to each other. Details are left to the reader since this is essentially the same argument as the classical case $d=1$.
\end{proof}

\section{Property FM}\label{s_kaz}

\begin{proof}[Proof of Proposition \ref{recfm}]
Let $\Gamma$ be a subgroup of $\Rec_d$ with Property FM. View $\Rec_d$ as acting on $\RR^d$ (identity outside $[0,1\mathclose[^d$) Property FM forces $\Gamma$ to be finitely generated (for the same reason as Property T), see \cite[Prop.\ 5.6]{cornulier2015irreducible}. Let $\Lambda$ be a dense finitely generated subgroup of $\RR^d$ containing all translations lengths of elements of $\Gamma$. Then $\Gamma$ preserves $\Lambda$ and acts faithfully on it. Let $\mathrm{Wob}(\Lambda)$ be the group of bounded displacement permutations of $\Lambda$ (where $\Lambda$ is viewed as set of vertices of its Cayley graph). Then this defines an injective homomorphism $\Gamma\to\mathrm{Wob}(\Lambda)$. Since the graph $\Lambda$ has uniform subexponential growth (uniformity being with respect to the choice of origin), the wobbling group $\mathrm{Wob}(\Lambda)$ contains no infinite subgroup with Property FM (this is \cite[Theorem 7.1(2)]{cornulier2015irreducible}, which follows the lines of \cite[Theorem 4.1]{juschenko2015invariant}, which asserts it for Property T).
\end{proof}

Note that this proof works equally for the whole group of permutations $f$ of $\RR^d$ such that $\{f(x)-x:x\in\RR^d\}$ is finite, and in particular works for the group of piecewise translations with arbitrary polyhedral pieces.

\section{A torsion group in $\Rec_3$}\label{torsion_section}

We make use of the following result of Nekrashevych \cite{Nek}\footnote{The theorem appeared in this way in a first preliminary ArXiv version of \cite{Nek} (v1) and was then generalized.}.

\begin{Thm}\label{Nekt}
Let $X$ be an infinite Stone space (=totally disconnected Hausdorff compact space) and $\xi\in X$. Let $a,b$ be self-homeomorphisms of $X$, with $a^2=b^2=\mathrm{id}_X$, with $b(\xi)=\xi$; assume that $\langle a,b\rangle$ acts minimally on $X$. Let $X_1,\dots,X_n$ be a $b$-invariant clopen partition of $X\smallsetminus\{\xi\}$ such that each $X_i$ accumulates at $\xi$. Let $b_i$ be the restriction of $b$ to $X_i$ (identity on $X\smallsetminus\{\xi\}$). Let $K\simeq (\ZZ/2\ZZ)^n$ be the subgroup generated by $b_1,\dots,b_n$. Let $H$ be a subgroup of $K$ not containing $b$, and all of whose $n$ projections are surjective. Then $\langle a,K\rangle$ is an infinite torsion group.
\end{Thm}

Note that there exists such a subgroup $H$ in $(\ZZ/2\ZZ)^n$ with the given conditions (avoiding the diagonal and with all projections surjective) if and only if $n\ge 3$, and then $H$ can be chosen to be of order 4 (e.g., generated by $b_1b_2$ and $b_2\dots b_n$).

To apply the theorem, it is convenient to work in the torus $T^3=\RR^3/\ZZ^3$ rather than $[0,1\mathclose[^3$: the definition of $\Rec_3^{\bowtie}(T^3)$ is immediate (using the canonical bijection $[0,1\mathclose[^3\to T^3$).

We start with two involutive self-homeomorphisms $a,b$ of $T$ given by $a(x)=v_0-x$ and $b(x)=-x$, where $v_0$ is a fixed totally irrational vector (in the sense that $\ZZ^3+\ZZ v_0$ is dense in $\RR^3$). Note that $\langle a,b\rangle$ acts minimally (since it contains a dense cyclic subgroup $\langle ab\rangle$ of translations of index 2).

Define $\bar{T}$ as the Denjoy-doubled circle: this is a copy of the circle in which each point $x$ has been replaced with a pair $\{x_-,x_+\}$. Endowed with the circular order, this is a Stone space, and the canonical two-to-one projection $\bar{T}\to T$ is continuous. Then each element of $\Rec_3^{\bowtie}(T^3)$ canonically lifts to a self-homeomorphism of $\bar{T}^3$. Hence, we obtain two involutive self-homeomorphisms $\bar{a},\bar{b}$ of $\bar{T}^3$, and $\langle\bar{a},\bar{b}\rangle$ acts minimally on $\bar{T}^3$.

We define a partition of $\bar{T}^3$ by cutting (in halves) the cube $[-1/2,1/2]^3$ into 8 cubes. Formally speaking: define $I_+=[0+,(1/2)_-]$, $I^-=[(-1/2)_+,0_-]$ and for any signs $a,b,c\in\{+,-\}$, define $C_{abc}=I_a\times I_b\times I_c$. Then define $C_0=C_{+++}\cup C_{---}$, $C_1=C_{-++}\cup C_{+--}$, $C_2=C_{+-+}\cup C_{-+-}$, $C_3=C_{++-}\cup C_{--+}$ (so $\bar{T}^3=C_0\sqcup C_1\sqcup C_2\sqcup C_3$).

\begin{Lem}\label{cantor}
There exists intermediate $\langle \bar{a},\bar{b}\rangle$-equivariant quotient map $\bar{T}^3\stackrel{p}\to K\stackrel{\pi}\to T^3$, such that $K$ is homeomorphic to a Cantor space, such that $\pi^{-1}(\{(0,0)\}$ is a singleton (denoted $0$), and such that the $p(C_i)$, $i=0,1,2,3$ are closed subsets pairwise intersecting at $0$.
\end{Lem} 

Granted the lemma, we conclude: defining $P_i=p(C_i)\smallsetminus\{0\}$, we obtain the desired clopen partition of $K\smallsetminus\{0\}$, and Theorem \ref{Nekt} applies. 

\begin{proof}[Proof of Lemma \ref{cantor}]
First, let $D_0,D_1$ be dense countable subset of $\RR$, each stable under all coordinate actions of $a$ and $b$ and by $x\mapsto x\pm 1$, with $0\in D_0$ and $0\notin D_1$. Write $D=D_0\cup D_1$.

Let $\bar{T}_D$ be the circle with all points in $D$ doubled (i.e., quotient of $\bar{T}$ by identifying $x_+$ and $x_-$ whenever $x\notin D$). This is a Cantor space. Then $a,b$ lift to $T_D^3$. Next, for every point $(x,y,z)$ in the $\langle a,b\rangle$-orbit of $(0,0,0)$, identify the 8 preimages of $(x,y,z)$ from $T_D^3$ to get a space $K$, and the quotient map $T_D^3\to K\to T^3$ are $\langle a,b\rangle$-equivariant.

It is enough to show that any two points in $K$ are separated by clopen subsets: this ensures that $K$ is both Hausdorff and totally disconnected. Write $p$ and $\pi$ for the projections as in the assertion of the lemma. If the two points have distinct images by $\pi$, this is straightforward: choose a small cube around one point with coordinates in $D_1$, small enough to avoid the other point.

Now suppose both points have the same image in $T^3$. Up to permute coordinates, we can suppose that these points have the form $(x_+,y',z')$ and $(x_-,y'',z'')$. Here either $y',y''$ are the same element of $T$, or have the form $y_+$ and $y_-$ for some $y$, similarly for $z',z''$. By assumption, $(x,y,z)$ is not in the orbit of zero. Since the orbit of zero is the orbit of powers of an irrational rotation, we deduce that no element $(x,y_1,z_1)$ closed enough to $(x,y,z)$ is in the orbit of zero. Hence, $[u,x_-]\times P$ and $[x_+,v]\times P$, for $u,v$ in $D_1$ close enough to $x$ and $P$ is a small enough 2-dimensional rectangle containing $(y,z)$, with coordinates, in $D_1$.

That the $p(C_i)$ are pairwise disjoint outside zero follows from the fact that the only element in the orbit of $(0,0,0)$ that has a 0 or 1 coordinate is $(0,0,0)$ itself.
\end{proof}

We have thus constructed an infinite finitely generated torsion subgroup in $\Rec_3^{\bowtie}$, and the latter embeds in $\Rec_3$ by Proposition \ref{ori_lift}.

\begin{Rem}
This construction is a variant of the one in \cite[\S 6.2]{Nek} (consisting of ``triangle exchanges''), which was suggested by the first-named author to V. Nekrashevych after reading a first version of \cite{Nek}. 
\end{Rem}
 
\bibliographystyle{plain}
\bibliography{biblio}
\end{document}